\documentclass[11pt,letterpaper,onecolumn]{article}
\usepackage{amsmath}
\usepackage{amssymb}
\usepackage{amsthm}
\usepackage{latexsym}
\usepackage{setspace}
\usepackage{enumerate}
\usepackage{graphicx,color}

\newtheorem{theorem}{Theorem}[section]
\newtheorem{proposition}[theorem]{Proposition}
\newtheorem{lemma}[theorem]{Lemma}
\newtheorem{corollary}[theorem]{Corollary}
\newtheorem{remark}[theorem]{Remark}

\newcommand{\topic}[1]{\par\bigskip\noindent{\bf #1}\nopagebreak\smallskip\noindent\par}

\parskip \medskipamount
\def\tsum{\mathop{\textstyle \sum }}%
\def\func#1{\mathop{\rm #1}\nolimits}%

\newcommand{\E}[1]{\operatorname{\mathbf{E}}\left[#1\right]}
\newcommand{\pr}[1]{\operatorname{\mathbf{P}}\left(#1\right)}
\newcommand{\IGNORE}[1]{}

\begin{document}

\title{\Large \textbf{Characterizing Optimal Sampling of Binary Contingency Tables via the
Configuration Model}}

\author{Jose~Blanchet\thanks{Department of Industrial Engineering and Operations Research, 
        Columbia University, 
        New York, NY 10027, U.S.A.\ \ Email: \hbox{blanchet@columbia.edu}.}
\and 
        Alexandre~Stauffer\thanks{Computer Science Division, University
        of California, Berkeley, CA~94720-1776, U.S.A.\ \ Email:
        \hbox{stauffer@cs.berkeley.edu}.}
}

\date{}

\maketitle

\begin{abstract}
A binary contingency table is an $m\times n$ array of binary entries with
row sums $\mathbf{r}=(r_{1},\ldots ,r_{m})$ and column sums $%
\mathbf{c}=(c_{1},\ldots ,c_{n})$. 
The configuration model generates a contingency table by 
considering $r_i$ tokens of type 1 for each row $i$ and $c_j$ tokens of type 2 for each column $j$,
and then taking a uniformly random pairing between type-1
and type-2 tokens.
We give a necessary and sufficient condition so that
the probability that the configuration model outputs a binary contingency
table remains bounded away from $0$ as $N=\sum_{i=1}^m r_i=\sum_{j=1}^n c_j$ goes to $\infty$. 
Our finding shows surprising
differences from recent results for binary \textit{symmetric} contingency
tables.
\end{abstract}

\section{Introduction}
\label{sec:intro}

Given two natural numbers $m$ and $n$, let $\mathbf{r}=(r_{1},r_{2},\ldots
,r_{m})$ and $\mathbf{c}=(c_{1},c_{2},\ldots ,c_{n})$ be vectors of positive
integers such that $\sum_{i=1}^{m}r_{i}=\sum_{j=1}^{n}c_{j}=N$. Let $%
\Omega _{\mathbf{r},\mathbf{c}}$ be the set of matrices with binary entries
such that the sum of the $i$-th row is given by $r_{i}$ and the sum of the $j
$-th column is given by $c_{j}$. These matrices are known as binary
contingency tables. We consider the problem of sampling uniformly from $%
\Omega _{\mathbf{r},\mathbf{c}}$ and of computing $|\Omega _{\mathbf{r},%
\mathbf{c}}|$. 

A binary contingency table can be used to represent the adjacency matrix of a bipartite 
graph. 
Therefore, the problem of sampling uniformly from $\Omega_{\mathbf{r},\mathbf{c}}$
is equivalent to uniformly sampling a bipartite graph with $m+n$ nodes such that the node degrees in one
partition are given by $r_1,r_2,\ldots,r_m$ and the node degrees in the other
partition are given by $c_1,c_2,\ldots,c_n$ (see Figure~\ref{fig:example}(a) and~(b) for an example).
\begin{figure}[tbp]
   \begin{center}
      \includegraphics[scale=.8]{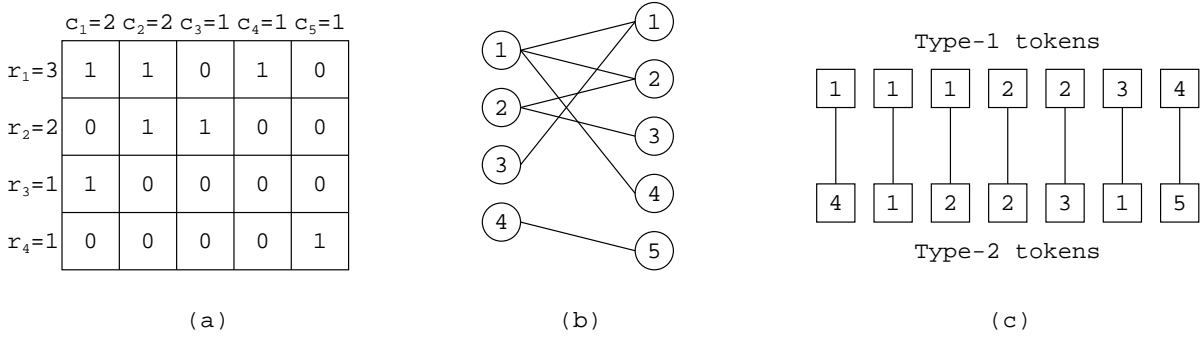}
   \end{center}
   \caption{(a)~A binary contingency table for the sequences $\mathbf{r}=\{3,2,1,1\}$ and $\mathbf{c}=\{2,2,1,1,1\}$. 
   (b)~The bipartite graph corresponding to the table of part~(a), where the leftmost partition represents the rows of the table and the rightmost partition
   represents the columns of the table. 
   (c)~A possible pairing between type-1 and type-2 tokens of the configuration model that corresponds to the table of part~(a). The labels in the type-1 and 
   type-2 tokens represent the index of the row and column, respectively, corresponding to that token.}
   \label{fig:example}
\end{figure}

We study the following well-known and simple algorithm for sampling contingency tables,
which is usually referred to as the
configuration model and was introduced by Bollob\'{a}s~\cite%
{bollobas1980}. 
For each row $i$,
consider $r_{i}$ tokens of type 1, and for each column $j$, consider $c_{j}$
tokens of type 2. Then, construct a table $T$ by sampling uniformly a random
matching between type-1 and type-2 tokens. In other words, 
first order the type-1 tokens in some arbitrary manner and draw a uniformly random permutation
of the type-2 tokens. Then, establish a matching between type-1 and
type-2 tokens according to the position in the permutation. In this way, the
entry $T_{i,j}$ is taken to be the number of type-1 tokens from row $i$ that
were matched to type-2 tokens from column $j$ (see Figure~\ref{fig:example}(c) for an example).
In the context of sampling bipartite graphs, the tokens 
are usually referred to as \emph{half-edges} and the matching establishes a pairing between 
half-edges of one partition and half-edges of the other partition.

The configuration model produces a table in $\Theta(N)$ time, but may output
a \emph{non-binary} table, which would correspond to a bipartite graph with more than one edge between the 
same pair of vertices. 
Yet, given that the table generated is binary, the
output table is a \emph{uniform} sample from $\Omega_{\mathbf{r},\mathbf{c}}$.
In order to see this, note that the table does not change if we take the permutation and switch the position of two type-2 tokens corresponding to the same column or 
if we switch the position of two type-2 tokens that are matched to type-1 tokens from the same row.
Thus, there are exactly $\prod_{i=1}^m\prod_{j=1}^n r_i!c_j!$ different permutations of type-2 tokens for any given binary
contingency table. Since this number does not depend on the table, but only on the sequences $\mathbf{r}$ and $\mathbf{c}$, we obtain that any given 
binary contingency table has the same probability to be generated by the configuration model.
Therefore, if the probability that the configuration model outputs a \emph{binary}
table does not go to zero as $N\to\infty$, we obtain both an
\emph{exact} sampler for the uniform distribution on $\Omega _{\mathbf{r,c}}$ and,
as explained in Section~\ref{sec:sketch}, a randomized algorithm to
approximate $\left\vert \Omega _{\mathbf{r,c}}\right\vert $ that runs in
time $\Theta(N)$. We call such a running time optimal for uniform generation
on $\Omega _{\mathbf{r,c}}$ since it takes at least $N$ elements to encode a
given binary table.

We study the asymptotic behavior as $N\rightarrow \infty $ of the
probability that the configuration model generates a binary table. 
For this reason, we consider \emph{input sequences} $(\mathbf{r}(N),\mathbf{c}(N))_{N\geq 1}$,
where for each $N\geq 1$, $\mathbf{r}(N)$ and $\mathbf{c}(N)$ are vectors
of cardinality $m(N)$ and $n(N)$, respectively, and whose elements are
non-negative integers and satisfy $\sum_{i=1}^{m(N)}r_{i}(N)=\sum_{j=1}^{n(N)}c_{j}(N)=N$. 
We assume that $r_{1}(N)\geq r_{2}(N)\geq\cdots \geq r_{m(N)}(N)$ and $c_{1}(N)\geq c_{2}(N)\geq \cdots \geq c_{n(N)}(N)$ for all $N$ and 
allow $m(N)$ and $n(N)$ to go to $\infty$ with $N$. It will be convenient in our development to consider the vectors $\mathbf{r}(N)$
and $\mathbf{c}(N)$, for fixed $N$, as having infinite elements. To this end, we
set $r_{i}(N)=0$ for $i\geq m(N)+1$ and $c_{j}(N)=0$ for $j\geq n(N)+1$. 
Without loss of generality we assume that $r_{1}(N)\geq c_{1}(N)$ for all $N$. 
Furthermore, for each $i$ and $j$, we regard $(r_{i})_{N\geq 1}$ and $(c_{j})_{N\geq 1}$ as sequences in their own right. 
Finally, when taking subsequences of $(\mathbf{r}(N),\mathbf{c}(N))_{N\geq 1}$, it is useful to see the input sequence as 
a sequence of tuples $(\mathbf{r}(\ell),\mathbf{c}(\ell),N(\ell))_{\ell\geq 1}$ with $N(\ell)=\ell$ for all $\ell$. With this, we have that 
\begin{align}
   &\text{a subsequence of 
   $(\mathbf{r}(\ell),\mathbf{c}(\ell),N(\ell))_{\ell\geq 1}$ is a sequence of tuples 
   $(\mathbf{r}'(\ell'),\mathbf{c}'(\ell'),N'(\ell'))_{\ell'\geq 1}$} \nonumber \\
   &\text{for which there exist positive integers $k_1 > k_2 >\cdots$ so that $\mathbf{r}'(\ell')\equiv \mathbf{r}(k_{\ell'})$,}\nonumber \\
   &\text{$\mathbf{c}'(\ell')\equiv \mathbf{c}(k_{\ell'})$ and 
   $\tsum_{i=1}^\infty r_i'(\ell')=\tsum_{j=1}^\infty c_j'(\ell')=N'(\ell')=N(k_{\ell'})$.}
   \label{eq:subsequence}
\end{align}
For brevity, 
we shall drop any explicit dependence on $\ell$ and $N$ from our notation; e.g., we refer to the input sequence as $(\mathbf{r},\mathbf{c})_{N\geq 1}$ and 
use the full notation $(\mathbf{r}(\ell),\mathbf{c}(\ell),N(\ell))_\ell$ only when talking about subsequences or where ambiguity may arise.

Our main result characterizes the class of input sequences $(\mathbf{r},\mathbf{c})_{N}$ for which the configuration model takes $\Theta(N)$ time to
sample uniformly from $\Omega _{\mathbf{r,c}}$ and to approximate $%
\left\vert \Omega _{\mathbf{r,c}}\right\vert $ as $N\to\infty$. Note that if $c_{1}=1$, all
the tables satisfying $\mathbf{r}$ and $\mathbf{c}$ are binary, so not only
the configuration model generates only binary contingency tables, but also $%
|\Omega _{\mathbf{r},\mathbf{c}}|$ can be trivially obtained. Thus, we
assume that $r_{1}\geq c_{1}\geq 2$ for all $N$. 
%

For any sequence of tuples $S=({\bf r}(\ell),{\bf c}(\ell),N(\ell))_{\ell\geq 1}$ we define 
$\kappa(S)$ to be the first row having sum $o(N)$;
more formally, 
\begin{equation}
   \kappa(S)=\min\{i\geq 1 \colon r_i(\ell)=o(N(\ell)) \text{ as $\ell\to\infty$}\}.
   \label{eq:kappa}
\end{equation}
We now state our main result, which gives necessary and sufficient conditions
for the optimality of the configuration model. 

\medskip
\begin{theorem}
\label{thm:main} Let $T$ be a table produced by the configuration model
given the input sequence $S=(\mathbf{r},\mathbf{c})_N$. 
We have that $\pr{T\in \Omega _{\mathbf{r},\mathbf{c}}}=\Omega(1)$ as $N\to\infty$ if and only if the following two conditions hold:
\begin{enumerate}
\item $\sum _{i=1}^{m}\sum _{j=1}^{n}r_i(r_i-1)c_j(c_j-1)=O(N^2)$.\label%
{it:n1}

\item For every subsequence $S'=({\bf r}'(\ell'),{\bf c}'(\ell'),N'(\ell'))_{\ell'}$ of $S$ 
we have that either 
$$
   \limsup_{\ell'\to\infty}\sum _{i=\kappa(S')}^\infty \frac{r'_i(\ell')}{N'(\ell')}>0
   \quad\text{ or }\quad
   \limsup_{\ell'\to\infty} c'_1(\ell')<\kappa(S').
$$
\label{it:n2}
\end{enumerate}
\end{theorem}

\medskip

\begin{remark}
\rm{
   We point out that $\kappa(S)$ may be $\infty$ (e.g., consider $r_i=\left\lfloor\tfrac{N}{2^i}\right\rfloor$ for all $i\in \{1,2,\ldots,\log_2 N\}$ and 
   $r_i=1$ for as many values of $i>\log_2N$ as needed to make the sum of the $r_i$'s equal to $N$). 
   In this case, we define $\sum _{i=\kappa(S)}^\infty \frac{r_i}{N}=0$.
}
\end{remark}

\begin{remark}
   \rm{
   When $r_1=o(N)$, condition~\ref{it:n2} above is always satisfied 
   since $\kappa(S')=1$ and $\sum_{i=\kappa(S')}^\infty r'_i(\ell')=N'(\ell')$ for every subsequence $S'=({\bf r}'(\ell'),{\bf c}'(\ell'),N'(\ell'))_{\ell'}$
   of $S$. Thus, condition~\ref{it:n1} is 
   both necessary and sufficient.
   }
\end{remark}

\begin{remark}
   \rm{
   Note that, for the practically relevant case where $(r_i)_N$ and $(c_j)_N$ are non-decreasing with $N$ for all $i$ and $j$, 
   condition~\ref{it:n2} can be replaced by the simpler condition 
   $$
   \sum _{i=\kappa(S)}^\infty r_i=\Omega(N)
   \quad\text{ or }\quad
   \lim_{N\to\infty} c_1<\kappa(S),
   $$
   which concerns only the sequence $({\bf r},{\bf c})_N$ and not every subsequence of~$({\bf r},{\bf c})_N$.
   }
\end{remark}
\begin{remark}
   \rm{
   Note that conditions~\ref{it:n1}~and~\ref{it:n2} 
   are \emph{not} redundant. 
   For instance, for any sequence $S=(\mathbf{r},\mathbf{c})_N$ with $r_1=N-o(N)$, $c_1=2$ and $c_j=1$ for all $j= 2,3,\ldots,N-1$, we have
   $\kappa(S) =2$, which violates condition~\ref{it:n2}, though
   condition~\ref{it:n1} holds. 
   }
\end{remark}
\medskip


Our theoretical developments are partly driven by our desire to guide
practitioners in areas of applied statistics who often deal with hypothesis
testing involving graphical models and binary contingency tables (see for
instance~\cite{besag1989} and \cite{chen2005}). In these types of settings,
data is encoded in the form of a binary table and one is interested in
studying the null hypothesis that row and column sums are sufficient
statistics for determining the distribution of all the entries in the table.
To test this hypothesis statisticians compare the value of a given statistic
of the observed table (e.g., the sum of the hamming distances of pairs of
rows) with values generated by sampling tables under the distribution
induced by the null hypothesis, which is precisely the uniform distribution
on binary contingency tables with prescribed row and column sums. For its simplicity and
small running time, the configuration model is a very appealing algorithm to be used 
in this setting. Our Theorem~\ref{thm:main} above fully characterizes the sequences
${\bf r}$ and ${\bf c}$
for which the configuration model is a fast and
reliable algorithm for uniform generation of binary contingency tables.

The configuration model is by now a classical, well-known algorithm that has 
been applied in practice, as described above, and also in more theoretical
settings.
For example, some asymptotic estimates for $|\Omega_{{\bf r},{\bf c}}|$ (e.g.,~\cite{mckay1984} 
and~\cite{greenhill2006}) are obtained via analyses of the configuration model.
Some results on the
structural properties of graphs obtained uniformly at random
from $\Omega_{{\bf r},{\bf c}}$ also use the configuration model
(e.g.,~\cite[Chapter~9]{janson2000} and~\cite{greenhill2008b}).
Usually, it is easier to analyze a graph obtained via the configuration model than
a random sample from $\Omega_{{\bf r},{\bf c}}$, and it is important to know whether 
results for one model can be carried over to the other. 
In order to explain how our results apply to this type of questions, 
let $A$ be any property
that can be tested for a bipartite graph (e.g., $A$ can be the property that the graph has a connected component with 
a constant fraction of the vertices, which is the property studied in~\cite{greenhill2008b}). If the conditions 
in Theorem~\ref{thm:main} hold, then any property $A$ that holds with probability $1-o(1)$ for the configuration model 
also holds with probability $1-o(1)$ for a graph obtained uniformly at random from $\Omega_{{\bf r},{\bf c}}$. 
This corresponds to the notion of contiguity between probability measures, 
which is more thoroughly explained in~\cite[Chapter~9]{janson2000}.
The corollary below gives an application of our results. 
We remark that this can only be obtained since the configuration model
is an \emph{exact} sampler for the uniform distribution over $\Omega_{{\bf r},{\bf c}}$.

\begin{corollary}\label{cor:contiguity}
   Let $A$ be a property that can be tested for a bipartite graph. Let $p(A)$ be the probability that a graph obtained 
   uniformly at random from $\Omega_{{\bf r},{\bf c}}$ contains property $A$, and $p'(A)$ be the probability that 
   a graph obtained via the configuration model given ${\bf r}$ and ${\bf c}$ contains property $A$. If 
   conditions~\ref{it:n1} and~\ref{it:n2} in Theorem~\ref{thm:main} are satisfied and $p'(A)=1-o(1)$, then $p(A)=1-o(1)$.
\end{corollary}
\begin{proof}
   Let $\rho$ be the probability that the configuration model outputs a binary table. Note that $\rho=\Omega(1)$ if 
   conditions~\ref{it:n1} and~\ref{it:n2} in Theorem~\ref{thm:main} are satisfied. Since the configuration model
   is an exact sampler for the uniform distribution over $\Omega_{{\bf r},{\bf c}}$, we obtain
   $p(A)\geq 1-\frac{1-p'(A)}{\rho}=1-o(1)$.
\end{proof}

\section{Related Work}\label{sec:related}
Theorem~\ref{thm:main} can be seen as an extension of recent work by Janson~%
\cite{janson2009}, who studied the probability that the configuration model
generates a binary \textit{symmetric} table. Letting $\Omega _{\mathbf{r}%
}^{\prime }$ be the set of all binary \textit{symmetric} tables with row and
column sums given by $\mathbf{r}$, \cite[Theorem 1.1]{janson2009}
establishes that $\pr{T\in \Omega _{%
\mathbf{r}}^{\prime }}=\Omega(1)$ if and only if $\sum_{i=1}^{m}r_{i}^{2}=O(N)$.

To contrast Janson's result to the case of non-symmetric tables studied
here, note that $(\mathbf{r},\mathbf{c})_{N}$ satisfying conditions~\ref%
{it:n1} and~\ref{it:n2} give rise to a much wider class of behavior than in
the symmetric case. For instance, the apparently similar conditions $%
\sum_{i=1}^{m}\sum_{j=1}^{n}r_{i}(r_{i}-1)c_{j}(c_{j}-1)=O(N^{2})$ and $%
\sum_{i=1}^{m}\sum_{j=1}^{n}r_{i}^{2}c_{j}^{2}=O(N^{2})$ are far from
identical; if $\mathbf{c}=\{2,1,1,\ldots ,1\}$, then the former condition is
satisfied regardless of $\mathbf{r}$ while the latter may not hold. 
Besides, the condition 
$\sum_{i=1}^{m}r_{i}^{2}=O(N)$ for symmetric tables allows $r_1$ to 
grow only as $O(\sqrt{N})$, whereas our Theorem~\ref{thm:main} 
reveals that there are sequences with $r_1$ as large as $N-o(N)$ for which
the configuration model produces a binary table with probability $\Omega(1)$. 
Therefore, the growth behavior allowed for $r_{1}$ in Theorem~\ref%
{thm:main} as $N\rightarrow \infty $ is much wider than in the symmetric
case. This
wider type of growth behavior makes the analysis for the non-symmetric case
qualitatively different. Moreover, our proof techniques are completely
different from those employed by Janson and
reveal some structural properties of the tables generated with
the configuration model. For
example, we show that conditioning on the entries with relatively large row
and column sums being binary, the probability that there is an entry with
value larger than $2$ is tiny (see Lemma~\ref{lem:truep3}). 
We believe that our techniques can be exploited in the analysis of related
problems (such as efficient sampling of non-binary contingency tables).

Polynomial-time algorithms have been developed
for the problem of approximating $|\Omega _{\mathbf{r},\mathbf{c}}|$.
In fact, approximating $|\Omega _{\mathbf{r},\mathbf{c}}|$
can be reduced to the problem of computing the permanent of a binary $\ell
\times \ell $ matrix with $\ell =\Theta (mn)$; a problem that enjoys a
notable history and place in the theory of computation. Valiant~\cite%
{valiant1979} showed that computing the permanent belongs to the class of
\#P-complete problems, for which proving the existence of a polynomial-time
algorithm would have extensive implications in complexity theory. It is
still an open problem, however, to verify whether counting the number of
binary contingency tables is \#P-complete, though the more general
problem of counting the number of (not necessarily binary) contingency
tables has been shown to be \#P-complete by Dyer et al.~\cite{dyer1997}. The
ground-breaking work of Jerrum et al.~\cite{jerrum2004} provided the first
Fully Polynomial Randomized Approximation Scheme (FPRAS)~\cite%
{mitzenmacher2005} to compute the permanent of a binary matrix. B\'{e}zakov%
\'{a} et al.~\cite{bezakova2008} used simulated annealing techniques to
develop an asymptotically faster algorithm to approximate the permanent,
which runs in $O(\ell ^{7}\log ^{4}\ell )$ time for an $\ell \times \ell $
matrix. In another paper, B\'{e}zakov\'{a} et al.~\cite{bezakova2007}
developed an algorithm that works directly with contingency tables. Their
algorithm for approximately sampling binary tables runs in $%
O(m^{2}n^{2}N^{3}\Delta \log ^{5}(m+n))$ time, where $\Delta $ is the
maximum over all row and column sums.

Although these algorithms are proved to run in polynomial time for all $%
\mathbf{r}$ and $\mathbf{c}$,
their efficiency is far from being useful in
the types of applications described at the end of Section~\ref{sec:intro}.
For this reason, other approaches to uniformly sampling and counting binary
contingency tables have been proposed. Chen et al.~\cite{chen2005} developed
a sequential importance sampling algorithm to count the number of
contingency tables. Their algorithm applies a heuristic construction and has
been observed to perform well in practice, 
but B\'{e}zakov\'{a} et al.~\cite{bezakova2006} proved that
there exist $\mathbf{r}$ and $\mathbf{c}$ such that the heuristic of Chen et
al.~\cite{chen2005} underestimates the number of binary contingency tables
by an exponential factor unless the algorithm is run for an exponential
amount of time. On the other hand, Blanchet~\cite{blanchet2009} provided a
rigorous analysis of the heuristic of Chen et al.~\cite{chen2005} and showed
that if $r_1=o(\sqrt{N})$, $\sum_{i=1}^{m}r_{i}^{2}=O(N)$, and 
$c_1=O(1)$, then this approach yields a FPRAS for counting
binary contingency tables with running time $O(N^{3})$. Our Theorem~\ref%
{thm:main} significantly weakens the assumptions in~\cite{blanchet2009}, and
drastically improves upon the running time of all the aforementioned
algorithms.

In a different direction, much effort has been made to derive asymptotics 
for $|\Omega _{\mathbf{r},\mathbf{c}}|$. 
The first result to allow the row and column sums to grow with $N$ is the one by O'Neil~\cite{oneil1969},
which is restricted to the case $n=m$ and $r_1=O(\log^{1/4-\epsilon}n)$ for any constant $\epsilon>0$.
Later, McKay~\cite{mckay1984} considered the case $r_1=o(N^{1/4})$ and derived the first asymptotics 
for $|\Omega _{\mathbf{r},\mathbf{c}}|$
to allow $r_1$ to 
grow polynomially with $N$. Currently, the asymptotics for \emph{sparse} binary tables that 
allows the largest range for $\mathbf{r}$ and $\mathbf{c}$ is the one by 
Greenhill et al.~\cite{greenhill2006} for the 
case $r_1c_1 = o(N^{2/3})$. 
These results by McKay~\cite{mckay1984} and Greenhill et al.~\cite{greenhill2006} were obtained 
using the 
configuration model as a part of their proof technique. Similarly,
the work of Blanchet discussed 
above~\cite{blanchet2009} also uses the configuration model, as well as McKay's estimator~\cite{mckay1984},
to analyze the heuristics of Chen et al.~\cite{chen2005}. 
Using different techniques, Canfield et al.~\cite{canfield2008} derived asymptotics for \emph{dense}
binary tables, and Barvinok~\cite{barvinok2010} derived general lower and upper bounds for 
$|\Omega _{\mathbf{r},\mathbf{c}}|$ that 
are within a factor $(mn)^{\Theta(m+n)}$ from each other.
For binary \textit{symmetric} tables, besides the work of Janson~\cite{janson2009}
cited above, we highlight the work of 
Bayati et al.~\cite{bayati2007}, who developed an
algorithm that generates a symmetric table almost uniformly at random in
time $O(r_{1}N)$ as long as $r_{1}=c_{1}=O(N^{1/4-\epsilon })$ for any
constant $\epsilon >0$. Their analysis gives an alternative proof of a result originally derived by McKay~\cite{mckay1985}.

We remark that, to the best of our knowledge, none of the existing asymptotics for $|\Omega _{\mathbf{r},\mathbf{c}}|$
applies to the whole of the spectrum of sequences $\mathbf{r}$ and $\mathbf{c}$ that satisfy 
conditions~\ref{it:n1} and~\ref{it:n2} in our Theorem~\ref{thm:main}.
Furthermore, most of the known asymptotics take advantage of the configuration model
in a fundamental way. Since our result fully characterizes the sequences for which the configuration model 
is contiguous to the uniform distribution, our conditions shed light into the whole spectrum of sequences for which 
analytical estimators might be obtained by directly applying the configuration model.

Under the conditions of Theorem~\ref{thm:main},
the configuration model gives a FPRAS for approximating $|\Omega _{\mathbf{r},\mathbf{c}}|$; thus
it approximates $|\Omega _{\mathbf{r},\mathbf{c}}|$ to a precision of the form 
$1+O(N^{-c})$, for an arbitrarily large constant $c>0$\footnote{We remark that under the conditions of Theorem~\ref{thm:main}
the configuration model approximates $|\Omega _{\mathbf{r},\mathbf{c}}|$ to a precision of the form $1\pm\epsilon$ for 
any \emph{constant} $\epsilon>0$ in time $\Theta(N)$, but can approximate $|\Omega _{\mathbf{r},\mathbf{c}}|$ to a precision
$1+ O(N^{-c})$ for an arbitrary constant $c>0$ in polynomial time.}, whereas asymptotics for 
$|\Omega _{\mathbf{r},\mathbf{c}}|$ have fixed 
precision. 
We remark that asymptotics for $|\Omega_{\mathbf{r},\mathbf{c}}|$  
can also be used to produce an \emph{almost} uniform sampling procedure for binary contingency tables.
Sinclair and Jerrum~\cite{sinclair1989} showed that for any self-reducible problem\footnote{Informally, a problem 
is self-reducible if it can be split in parts where each part is itself a smaller instance of the 
same problem. In the case of sampling binary contingency tables, after generating all the entries of a given column,
we can update the row and column sums properly so that generating the remaining entries translates to sampling 
a binary contingency table with different row and column sums.}
an asymptotic approximation with at least constant 
precision can be used to produce an almost uniform sampling procedure.
However, the running time of the sampling procedure 
depends on the mixing time of a Markov chain, 
which not only may be challenging to obtain precisely but also is usually too large for many practical applications. 
Moreover, we remark that this technique
cannot be directly employed with 
the current asymptotics for $|\Omega _{\mathbf{r},\mathbf{c}}|$ since they impose some conditions on $\mathbf{r}$ and 
$\mathbf{c}$. Under these conditions, the problem of sampling binary contingency tables is not guaranteed to be self-reducible: 
when splitting the table into smaller tables, we do not necessarily obtain that the new row and column sums satisfy the conditions 
of the asymptotic results. 

\section{Preliminaries}
\label{sec:sketch}

As mentioned in Section~\ref{sec:intro}, we use the configuration model to
generate a contingency table $T$ (not necessarily binary). There are $N!$
possible matchings among the tokens, but any given \textit{binary}
contingency table generated by the configuration model corresponds to $%
\prod_{i=1}^{m}\prod_{j=1}^{n}r_{i}!c_{j}!$ such matchings, since permuting
the tokens within each row or column does not change the final table.
Therefore, we can conclude that $|\Omega _{\mathbf{r},\mathbf{c}%
}|\prod_{i=1}^{m}\prod_{j=1}^{n}r_{i}!c_{j}!=\pr{T\in \Omega _{\mathbf{r},\mathbf{c}}}N!$, 
and the problem of computing $|\Omega _{\mathbf{r},\mathbf{c}}|$ is equivalent to evaluating 
$\pr{T\in \Omega _{\mathbf{r},\mathbf{c}}}$.

If $\pr{T\in \Omega _{\mathbf{r},\mathbf{c}}}=\Omega(1)$,
then we obtain a Fully Polynomial Randomized Approximation Scheme (FPRAS)
for estimating $|\Omega _{\mathbf{r},\mathbf{c}}|$ as follows (we refer the
reader to \cite{mitzenmacher2005} for more information on FPRAS). Generate a
sequence of independent contingency tables using the configuration model and
output the fraction of the tables that turn out to be binary. If $%
\pr{T\in \Omega _{\mathbf{r},\mathbf{c}}}=\Omega(1)$, for
any constants $\delta >0$ and $\epsilon >0$, it suffices to generate a
constant (depending polynomially on $\epsilon ^{-1}$ and $\log \delta ^{-1}$%
) number of tables such that with probability $1-\delta $ our estimator to $%
|\Omega _{\mathbf{r},\mathbf{c}}|$ has precision $1\pm \epsilon $.

We conclude this section by introducing fundamental notation that we will use in the proof.
Let $I$ be the index set $[1,m]\times \lbrack 1,n]$ and $T=(T_{i,j})_{(i,j)%
\in I}$ be a table generated by the configuration model. 
Let $Z$ be the number of non-binary entries of $T$, so $\pr%
{T\in \Omega _{\mathbf{r},\mathbf{c}}}=\pr{Z=0}$. Given two integers $%
k\geq 0$ and $x\geq 0$, we define $x^{\underline{k}}=\frac{x!}{(x-k)!}$.
Recall that the configuration model generates a table by taking a random
matching between type-1 and type-2 tokens. We assume that each token is
individually labeled and refer to a single pair of a type-1 and a type-2
token as an \textit{edge}. We say that an edge is matched by the
configuration model if the corresponding tokens are matched. A set of two
edges for the same entry is referred to as a \textit{double edge}. For $%
(i,j)\in I$, let ${\mathcal{B}}_{2}(i,j)$ be the set of all possible double
edges that can be matched for the entry $(i,j)$. An element of ${\mathcal{B}}%
_{2}(i,j)$ has the form $\{e_{1},e_{2}\} $, where $e_{1}$ and $e_{2}$ are
disjoint edges for the entry $(i,j)$, that is, $e_{1}$ and $e_{2}$
correspond to $4$ distinct tokens, $2$ type-1 tokens from row $i$ and $2$
type-2 tokens from column $j$. Clearly, the cardinality of ${\mathcal{B}}%
_{2}(i,j)$ is given by $|{\mathcal{B}}_{2}(i,j)|=r_{i}^{\underline{2}%
}c_{j}^{\underline{2}}/2!$. For any $(i,j)\in I$ and $B\in {\mathcal{B}}%
_{2}(i,j)$, let ${\mathcal{M}}(B)$ be the event that the double edge
represented by $B$ is matched by the configuration model. Note that given
any specific $B\in {\mathcal{B}}_{2}(i,j)$, $\pr{{\mathcal{M}}%
(B)}=1/N^{\underline{2}}$. With this notation, note that the event $\{Z\geq
1\}$ is equivalent to $\{\bigcup_{B\in {\mathcal{B}}_{2}}{\mathcal{M}}(B)\}$.

\section{Proof of Theorem~\ref{thm:main}}
The proof of Theorem~\ref{thm:main} follows from the three propositions below, which
we will prove in subsequent sections. 
The first proposition, Proposition~\ref{prop:nec}, shows that 
condition~\ref{it:n1} is necessary; its proof is given 
in Section~\ref{sec:nec}.

\begin{proposition}
\label{prop:nec} As $N\to\infty$, if $\sum_{(i,j)\in I}r_{i}(r_{i}-1)c_{j}(c_{j}-1)$ is not $O(N^2)$, then 
$\pr{T\in \Omega _{\mathbf{r},\mathbf{c}}}$ is not $\Omega(1)$.
\end{proposition}
\medskip

The proof of Proposition~\ref{prop:nec} highlights the importance of the definition of \textit{double edges},
since condition~\ref{it:n1} in Theorem~\ref{thm:main} translates
to the expected number of double edges in $T$ being uniformly bounded over $%
N$. Note that for the case of symmetric tables, condition~\ref{it:n1} is both necessary and 
sufficient, while for the non-symmetric case it is just necessary. 
Now, we assume that $r_{1}=o(N)$ and
show in Proposition~\ref{prop:suf} that, in this case, condition~\ref{it:n1} in Theorem~\ref{thm:main} is
also sufficient. The proof of Proposition~\ref{prop:suf} is presented in Section~\ref{sec:suf}.

\begin{proposition}
\label{prop:suf} As $N\to\infty$, if $\sum_{(i,j)\in I}r_{i}(r_{i}-1)c_{j}(c_{j}-1)=O(N^{2})$
and $r_{1}=o(N)$ then 
$\pr{T\in \Omega_{\mathbf{r},\mathbf{c}}}=\Omega(1)$.
\end{proposition}
\medskip

If $r_1$ is not $o(N)$, i.e., $\limsup_{N\to\infty} r_1/N>0$, then $(\mathbf{%
r},\mathbf{c})_N$ contains a subsequence $(\mathbf{r}^\prime(\ell'),\mathbf{c}%
^\prime(\ell'),N'(\ell'))_{\ell\geq 1}$ for which $r_1^{\prime }(\ell')=\Omega(N'(\ell'))$ as $\ell'\to\infty$. 
(Recall the definition of a subsequence of an input sequence in~\eqref{eq:subsequence}.)
The next proposition deals with
the case $r_1=\Omega(N)$ and its proof is presented in Section~\ref%
{sec:extra}.

\begin{proposition}
\label{prop:extra} As $N\to\infty$, if $\sum_{(i,j)\in I}r_{i}(r_{i}-1)c_{j}(c_{j}-1)=O(N^{2})$
and $r_{1}=\Omega(N)$, then $\pr{T\in\Omega _{\mathbf{r},\mathbf{c}}}=\Omega(1)$ 
if and only if, for every subsequence $S'=({\bf r}'(\ell'),{\bf c}'(\ell'),N'(\ell'))_{\ell'\geq 1}$ of 
$S=({\bf r},{\bf c})_N$, we have
$\limsup_{\ell'\to\infty}\sum _{i=\kappa(S')}^\infty \frac{r'_i(\ell')}{N'(\ell')}>0$
or $\limsup_{\ell'\to\infty} c'_1(\ell')<\kappa(S')$,
where~$\kappa$ is defined as in~\eqref{eq:kappa}.
\end{proposition}
\medskip

It is clear that Propositions~\ref{prop:nec},~\ref{prop:suf}, and~\ref{prop:extra} establish 
that $\pr{T\in\Omega_{\mathbf{r},\mathbf{c}}}=\Omega(1)$ if and only if both condition~\ref{it:n1} 
and~\ref{it:n2} in
Theorem~\ref{thm:main} are satisfied and $r_1$ is either $o(N)$ or $\Omega(N)$. 
We now explain the case when $r_1$ is neither $o(N)$ nor $\Omega(N)$, i.e., 
$\limsup_{N\to\infty}r_1/N>\liminf_{N\to\infty}r_1/N=0$. 
For this, we will make use of the following technical lemma, 
which is also used in~\cite[chapter 9]{janson2000} and~\cite{janson2009}.

\begin{lemma}[Subsubsequence principle]
\label{lem:subsubsequence}
   If every subsequence $(\mathbf{r}'(\ell'),\mathbf{c}'(\ell'),N'(\ell'))_{\ell'}$
   of $(\mathbf{r},\mathbf{c})_N$
   contains a further subsequence $(\mathbf{r}''(\ell''),\mathbf{c}''(\ell''),N''(\ell''))_{\ell''}$ for which
   $\pr{T\in\Omega _{\mathbf{r}'',\mathbf{c}''}}=\Omega(1)$ as $\ell''\to\infty$, then 
   $\pr{T\in\Omega _{\mathbf{r},\mathbf{c}}}=\Omega(1)$ as $N\to\infty$.
\end{lemma}
\begin{proof}
   We will prove this lemma by contradiction. Assume that 
   every subsequence $(\mathbf{r}'(\ell'),\mathbf{c}'(\ell'),N'(\ell'))_{\ell'}$
   of $(\mathbf{r},\mathbf{c})_N$
   contains a further subsequence $(\mathbf{r}''(\ell''),\mathbf{c}''(\ell''),N''(\ell''))_{\ell''}$ for which
   $\pr{T\in\Omega _{\mathbf{r}'',\mathbf{c}''}}=\Omega(1)$ as $\ell''\to\infty$, but 
   $\pr{T\in\Omega _{\mathbf{r},\mathbf{c}}}$ is not $\Omega(1)$ as $N\to\infty$. This means that 
   $\liminf_{N\to\infty}\pr{T\in\Omega _{\mathbf{r},\mathbf{c}}}=0$ and, consequently,
   there exists a subsequence $(\hat{\mathbf{r}}(\hat \ell),\hat{\mathbf{c}}(\hat \ell),\hat N(\hat \ell))_{\hat \ell}$ of $(\mathbf{r},\mathbf{c})_N$
   such that $\lim_{\hat \ell\to\infty}\pr{T\in\Omega _{\hat{ \mathbf{r}},\hat {\mathbf{c}}}}=0$, which contradicts the assumption.   
\end{proof}

Finally, we will be able to conclude the proof of Theorem~\ref{thm:main} with Lemma~\ref{lem:applyssseqp} below, which uses 
the subsubsequence principle. 
Since we will also apply this lemma later on, 
we give it in more generality than needed here. To avoid ambiguity,
we use the full notation discussed in~\eqref{eq:subsequence} and in the paragraph immediately preceding it.
\begin{lemma}
\label{lem:applyssseqp}
   Let $\mathcal{Z}$ be some space of sequences $(\mathbf{r}(\ell),\mathbf{c}(\ell),N(\ell))_\ell$ indexed by $\ell\geq 1$ such that if 
   $(\mathbf{r}(\ell),\mathbf{c}(\ell),N(\ell))_\ell\in\mathcal{Z}$ then any subsequence of $(\mathbf{r}(\ell),\mathbf{c}(\ell),N(\ell))_\ell$ is also in $\mathcal{Z}$. 
   Given a sequence $S=(\mathbf{r}(\ell),\mathbf{c}(\ell),N(\ell))_\ell\in\mathcal{Z}$, 
   let $f_\ell(S)$ be a sequence of nonnegative real numbers indexed by 
   $\ell\geq 1$.
   Assume that for all $S=(\mathbf{r}(\ell),\mathbf{c}(\ell),N(\ell))_\ell\in\mathcal{Z}$ such that $f_\ell(S)=\Omega(1)$ or
   $f_\ell(S)=o(1)$ as $\ell\to\infty$, we have $\pr{T \in \Omega_{\mathbf{r},\mathbf{c}}}=\Omega(1)$ as $N\to\infty$.
   Then, $\pr{T \in \Omega_{\mathbf{r},\mathbf{c}}}=\Omega(1)$ also holds for each $S=(\mathbf{r}(\ell),\mathbf{c}(\ell),N(\ell))_\ell\in\mathcal{Z}$ for 
   which
   $\limsup_{\ell\to\infty} f_\ell(S) > \liminf_{\ell\to\infty}f_\ell(S)=0$.
\end{lemma}
\begin{proof}
   We use the subsubsequence principle (Lemma~\ref{lem:subsubsequence}). 
   If $\limsup_{\ell\to\infty} f_\ell(S) > \liminf_{\ell\to\infty}f_\ell(S)=0$, then 
   for every subsequence $S'=(\mathbf{r}'(\ell'),\mathbf{c}'(\ell'),N'(\ell'))_{\ell'}$ of $S$ it is the case that 
   $S'\in\mathcal{Z}$ (by the property of $\mathcal{Z}$) and either 
   \begin{equation}
      \limsup_{\ell'\to\infty} f_{\ell'}(S')=0
      \label{eq:subspart1}
   \end{equation}
   or 
   \begin{equation}
      \text{there is a subsequence $S''=(\mathbf{r}''(\ell''),\mathbf{c}''(\ell''),N''(\ell''))_{\ell''}\in\mathcal{Z}$ of $S'$ for which
      $\liminf_{\ell''\to\infty}f_{\ell''}(S'')>0$.}
      \label{eq:subspart2}
   \end{equation}
   In the case of~\eqref{eq:subspart1}, since $f_{\ell'}(S')=o(1)$ as $\ell'\to\infty$, 
   we know that $\pr{T \in \Omega_{\mathbf{r'},\mathbf{c'}}}=\Omega(1)$. In the case of~\eqref{eq:subspart2},  
   since $f_N(S'')=\Omega(1)$, we have $\pr{T \in \Omega_{\mathbf{r''},\mathbf{c''}}}=\Omega(1)$ as $\ell''\to\infty$.
   Therefore, using the subsubsequence principle we obtain 
   $\pr{T \in \Omega_{\mathbf{r},\mathbf{c}}}=\Omega(1)$ as $N\to\infty$.
\end{proof}

We set $\mathcal{Z}$ as the space 
of sequences satisfying conditions~\ref{it:n1} and~\ref{it:n2} from Theorem~\ref{thm:main}. It is easy to check that 
this space satisfies the condition in Lemma~\ref{lem:applyssseqp}. Then, for $S=(\mathbf{r}(\ell),\mathbf{c}(\ell),N(\ell))_\ell\in\mathcal{Z}$ we set 
$f_\ell(S)=\tfrac{r_1(\ell)}{N(\ell)}$, and Lemma~\ref{lem:applyssseqp} gives that, as $N\to\infty$,
$\pr{T\in\Omega_{\mathbf{r},\mathbf{c}}}=\Omega(1)$ also for the 
case when $r_1$ is neither $\Omega(N)$ nor $o(N)$ but conditions~\ref{it:n1} and~\ref{it:n2} hold. 
This completes the proof of Theorem~\ref{thm:main}.

\section{Proof of Proposition~\protect\ref{prop:nec}}
\label{sec:nec}

We prove Proposition~\ref{prop:nec} using the second-moment method.
We define the function 
\begin{equation}
\mu(N)=\sum_{(i,j)\in I}\frac{r_{i}^{\underline{2}}c_{j}^{\underline{2}}}{%
2N^{\underline{2}}},  \label{eq:cond2}
\end{equation}
which satisfies $\limsup_{N\to\infty} \mu(N)=\infty$ by the assumptions of 
Proposition~\ref{prop:nec}, and show that under this condition
$\liminf_{N\to\infty}\pr{T\in \Omega_{\mathbf{r},\mathbf{c}}}=0$ (i.e., 
$\pr{T\in \Omega_{\mathbf{r},\mathbf{c}}}$ is not $\Omega(1)$). 
Let $F$ be the random
variable counting the number of double edges that are matched by the
configuration model, that is, $F=\sum_{B\in {\mathcal{B}}_{2}}{\mathbf{1}}%
\left( {\mathcal{M}}(B)\right) $, where ${\mathbf{1}}\left( \cdot \right) $
is the indicator function. Note that 
\begin{equation}
{\mathbf{E}}F=\sum_{(i,j)\in I}\sum_{B\in {\mathcal{B}}_{2}(i,j)}\frac{1}{N^{%
\underline{2}}}=\sum_{(i,j)\in I}\frac{r_{i}^{\underline{2}}c_{j}^{%
\underline{2}}}{2!N^{\underline{2}}}=\mu (N).  \label{eq:exp}
\end{equation}%
Our strategy is to use Chebyshev's inequality to obtain an upper
bound for $\pr{T\in \Omega _{\mathbf{r},\mathbf{c}}}$ via 
\begin{equation}
\pr{T\in \Omega _{\mathbf{r},\mathbf{c}}}=\pr{F\leq 0}=\pr{%
{\mathbf{E}}F-F\geq {\mathbf{E}}F}\leq \frac{{\mathbf{Var}}(F)}{{\mathbf{E}}%
^{2}F}.  \label{eq:cheb}
\end{equation}

We now derive an upper bound for ${\mathbf{Var}}(F)$. Note that 
\begin{equation*}
{\mathbf{E}}F^{2}=\sum_{(i,j)\in I}\sum_{B\in {\mathcal{B}}%
_{2}(i,j)}\sum_{(i^{\prime },j^{\prime })\in I}\sum_{B^{\prime }\in {%
\mathcal{B}}_{2}(i^{\prime },j^{\prime })}\pr{{\mathcal{M}}%
(B)\cap {\mathcal{M}}(B^{\prime })} ,
\end{equation*}%
from which we can write 
\begin{eqnarray}
{\mathbf{Var}}(F) &=&\sum_{(i,j)\in I}\sum_{B\in {\mathcal{B}}_{2}(i,j)}%
\pr{{\mathcal{M}}(B)}\sum_{(i^{\prime },j^{\prime })\in
I}\sum_{B^{\prime }\in {\mathcal{B}}_{2}(i^{\prime },j^{\prime })}\left( 
\pr{{\mathcal{M}}(B^{\prime })\mid {\mathcal{M}}(B)}-\pr{{%
\mathcal{M}}(B^{\prime })}\right)  \label{eq:varexp}\\
&\leq &\sum_{(i,j)\in I}\frac{r_{i}^{\underline{2}}c_{j}^{\underline{2}}}{%
2N^{\underline{2}}}\left( 1+\varphi _{1}(i,j)+\varphi _{2}(i,j)+\varphi
_{3}(i,j)+\varphi _{4}(i,j)\right) ,\nonumber
\end{eqnarray}%
where the terms $1$ and $\varphi _{1}(i,j)$ to $\varphi _{4}(i,j)$ are
explained next. First of all, when $(i,j)=(i^{\prime },j^{\prime })$ and $%
B=B^{\prime }$ we have that $\left( \pr{{\mathcal{M}}(B^{\prime })\mid 
{\mathcal{M}}(B)}-\pr{{\mathcal{M}}(B^{\prime })}\right) \leq 1$. Now,
the term $\varphi _{1}(i,j)$ corresponds to the cases where $B$ and $%
B^{\prime }$ are double edges for the same entry $(i,j)$ and also have one
edge in common (i.e., $B\cup B^{\prime }$ is a set of three edges). In such
cases, to compute $\varphi _{1}(i,j)$ we shall use $\pr{{\mathcal{M}}%
(B^{\prime })\mid {\mathcal{M}}(B)}-\pr{{\mathcal{M}}(B^{\prime
})}\leq \pr{{\mathcal{M}}(B^{\prime })\mid {\mathcal{M}}(B)}$ and
simply estimate $\pr{{\mathcal{M}}(B^{\prime })\mid {\mathcal{M}}(B)}$%
. The term $\varphi _{2}(i,j)$ corresponds to the terms where $B$ and $%
B^{\prime }$ are double edges for the same entry $(i,j)$ but have no edge in
common (i.e., $B\cup B^{\prime }$ is a set of four edges). Before proceeding
to describe $\varphi _{3}(i,j)$ and $\varphi _{4}(i,j)$, let us explain how
to compute $\varphi _{1}(i,j)$ and $\varphi _{2}(i,j)$, which we express as%
\begin{equation}
\varphi _{1}(i,j)=\frac{2(r_{i}-2)(c_{j}-2)}{N-2},  \label{eq:phi1}
\end{equation}%
and 
\begin{equation}
\varphi _{2}(i,j)=\frac{(r_{i}-2)^{\underline{2}}(c_{j}-2)^{\underline{2}}}{%
2(N-2)^{\underline{2}}}-\frac{r_{i}^{\underline{2}}c_{j}^{\underline{2}}}{%
2N^{\underline{2}}} \leq 0.  \label{eq:phi2}
\end{equation}%
Given that a double edge $\{e_{1},e_{2}\}$ is chosen from the entry $(i,j)$, 
$\varphi _{1}(i,j)$ is the probability that another edge $e_{3}$ from $(i,j)$
is chosen, which is given by $\frac{(r_{i}-2)(c_{j}-2)}{N-2}$. The
additional factor $2$ comes from the fact that, once we fix $e_{3}$, there
are $2$ possible choices of double edges for $B^{\prime }$, namely $%
B^{\prime }=\{e_{1},e_{3}\}$ and $B^{\prime }=\{e_{2},e_{3}\}$. The equation
for $\varphi _{2}(i,j)$ is obtained in a similar way, but we need to compute
the probability that we choose a double edge $\{e_{3},e_{4}\}$ from $(i,j)$
such that $\{e_{1},e_{2}\}\cap \{e_{3},e_{4}\}=\emptyset $, which gives the
term $\frac{(r_{i}-2)^{\underline{2}}(c_{j}-2)^{\underline{2}}}{2(N-2)^{%
\underline{2}}}$. The last term in (\ref{eq:phi2}) comes from the term $\pr{\mathcal{M}(B')}$ from~\eqref{eq:varexp},
which is the probability that a double edge $%
\{e_{3},e_{4}\}$ is chosen independently of $\{e_{1},e_{2}\}$.

The term $\varphi _3(i,j)$ corresponds to the terms where $B$ and $B^{\prime
}$ are double edges for the same row but different columns or for the same
column but different rows. Using similar reasoning we obtain 
\begin{equation}
\varphi _3(i,j)=\sum _{i^{\prime }\neq i}\left (\frac{r_{i^{\prime
}}^{\underline 2}(c_j-2)^{\underline 2}}{2(N-2)^{\underline 2}}-\frac{%
r_{i^{\prime }}^{\underline 2}c_j^{\underline 2}}{2N^{\underline 2}}\right
)+\sum _{j^{\prime }\neq j}\left (\frac{(r_{i}-2)^{\underline 2}c_{j^{\prime
}}^{\underline 2}}{2(N-2)^{\underline 2}}-\frac{r_{i}^{\underline
2}c_{j^{\prime }}^{\underline 2}}{2N^{\underline 2}}\right )\leq0,
\label{eq:phi3}
\end{equation}
Ultimately, $\varphi _4(i,j)$ corresponds to terms where $B$
and $B^{\prime }$ represent disjoint rows and columns and is given by 
\begin{equation}
\varphi _4(i,j)=\sum _{i^{\prime }\neq i}\sum _{j^{\prime }\neq j}\left (%
\frac{r_{i^{\prime }}^{\underline 2}c_{j^{\prime }}^{\underline 2}}{%
2(N-2)^{\underline 2}}-\frac{r_{i^{\prime }}^{\underline 2}c_{j^{\prime
}}^{\underline 2}}{2N^{\underline 2}}\right ).
\end{equation}

Now we simplify the equations above. For $\varphi_2(i,j)$ and $\varphi_3(i,j)$ we simply
use the fact that they are at most $0$. For $\varphi_1(i,j)$, we have that 
(\ref{eq:cond2})
implies 
\begin{equation}
\varphi _{1}(i,j)\leq 2\frac{\sqrt{r_{i}(r_{i}-1)c_{j}(c_{j}-1)}}{N-2}\leq 2%
\sqrt{2\mu (N)}(1+O(1/N)),  \label{eq:newphi1}
\end{equation}%
for $r_{i},c_{j}\geq 2$. Clearly, the entries $(i,j)$ with $r_{i}\leq 1$ or $%
c_{j}\leq 1$ do not contribute to $F$. 
Since $1/(N-2)^{\underline{2}}-1/N^{\underline{2}}=O(1/N^{3})$ we can write 
\begin{equation*}
\varphi _{4}(i,j)\leq \sum_{i^{\prime }\neq i}\sum_{j^{\prime }\neq j}\frac{%
r_{i^{\prime }}^{\underline{2}}c_{j^{\prime }}^{\underline{2}}}{2}%
O(1/N^{3})=O(\mu (N)/N).
\end{equation*}

Therefore, the variance of $F$ translates to 
\begin{eqnarray}
{\mathbf{Var}}(F) &\leq &\sum_{(i,j)\in I}\frac{r_{i}^{\underline{2}}c_{j}^{%
\underline{2}}}{2N^{\underline{2}}}\left( 1+2\sqrt{2\mu (N)}+O(\mu
(N)/N)\right)  \notag \\
&=&\mu (N)O(\sqrt{\mu (N)}+\mu (N)/N).  \label{eq:var}
\end{eqnarray}%
Since $r_{i},c_{j}\leq N$ for all $(i,j)\in I$, we have $\mu (N)=O(N^{2})$,
and consequently, $\sqrt{\mu (N)}=\Omega (\mu (N)/N)$. Now, plugging (\ref%
{eq:exp}) and (\ref{eq:var}) into (\ref{eq:cheb}), we obtain that there are
constants $C$ and $N_0$ such that for all $N\geq N_0$ we obtain 
\begin{equation*}
\pr{T\in \Omega _{\mathbf{r},\mathbf{c}}}=\pr{F\leq 0} \leq 
\frac{C}{\sqrt{\mu (N)}}+\frac{C}{N}.
\end{equation*}
Taking the $\liminf$ as $N\to\infty$ concludes the proof of Proposition~\ref%
{prop:nec} since $\limsup_{N\to\infty}\mu(N)=\infty$.

\section{Proof of Proposition~\protect\ref{prop:suf}}
\label{sec:suf}

In the proof of Proposition~\ref{prop:suf}
we assume that $2\leq c_{1}\leq r_{1}=o(N)$ and that there exists a constant $C>0$ such that 
for all sufficiently large $N$ it holds that
\begin{equation}
\sum_{(i,j)\in I}\frac{r_{i}^{\underline{2}}c_{j}^{\underline{2}}}{2N^{\underline{2}}} \leq C.
\label{eq:truesuf}
\end{equation}

We split the table $T$ into two regions $I_\mathrm{L}$ and $I_\mathrm{S}$.
Let $\epsilon>0$ be a small number that we will set later, and 
define the index sets $I_{\mathrm{L}}=\{(i,j)\in
I\colon (r_{i}-1)(c_{j}-1)\geq \epsilon N\}$ and $I_\mathrm{S}=I\setminus I_%
\mathrm{L}$. 
We remark that the set $I_{\mathrm{L}}$ may be empty.
Intuitively, $I_{\mathrm{L}}$ represents the entries of $T$
with large row and column sums. Since the $r_{i}$'s and the $c_{j}$'s are
assumed to be non-increasing, a useful conceptual diagram for the definition
of $I_{\mathrm{L}}$ is given by the shaded area in Figure~\ref{fig:diag}(a).

\begin{figure}[tbp]
   \begin{center}
      \includegraphics[scale=.9]{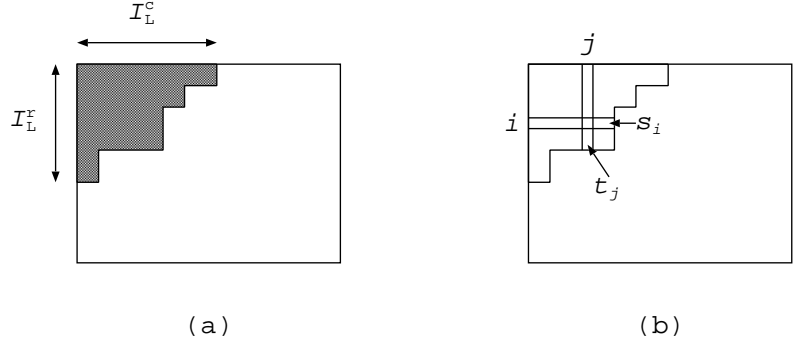}
   \end{center}
   \caption{Diagram illustrating the entries in $I_{\mathrm{L}}$ (shaded area in part~(a)) and the
   definition of $I_\mathrm{L}^\mathrm{r}$, $I_\mathrm{L}^\mathrm{c}$, $s_i$,
   and $t_j$.}
   \label{fig:diag}
\end{figure}

Let $Z_{\mathrm{L}}$ be the number of non-binary entries in $I_{\mathrm{L}}$%
, and $Z_{\mathrm{S}}$ be the number of non-binary entries in $I_{\mathrm{S}%
} $. Clearly, $Z=Z_{\mathrm{L}}+Z_{\mathrm{S}}$. Let $W_\mathrm{L}$ be the sum of the entries in 
$I_\mathrm{L}$ (i.e., $W_\mathrm{L}=\sum_{(i,j)\in I_\mathrm{L}}T_{i,j}$). 
Note that $\{W_\textrm{L}=0 \} \subseteq \{Z_\textrm{L}=0\}$, which gives
$\pr{T\in \Omega _{%
\mathbf{r},\mathbf{c}}} \geq \pr{W_{\mathrm{L}}=0}\pr{Z_{\mathrm{S}%
}=0\mid W_{\mathrm{L}}=0}$. We will deal with the terms 
$\pr{W_{\mathrm{L}}=0}$ and $\pr{Z_{\mathrm{S}%
}=0\mid W_{\mathrm{L}}=0}$ separately.
We will need the following definitions.
Let $I_{\mathrm{L}}^{\mathrm{r}}$ be the set of rows with at least one entry
in $I_{\mathrm{L}}$ (i.e., $I_{\mathrm{L}}^{\mathrm{r}}=\{i\colon (i,j)\in
I_{\mathrm{L}}$ for some $j\})$, and $I_{\mathrm{L}}^{\mathrm{c}}$
be the set of columns with at least one entry in $I_{\mathrm{L}}$ (see Figure~\ref{fig:diag}(a)).
Lemma~\ref{lem:truezl} below deals with the term 
$\pr{W_{\mathrm{L}}=0}$. 

\begin{lemma}
\label{lem:truezl} 
If (\ref{eq:truesuf}) is satisfied and $r_1=o(N)$, then we
obtain $|I_\mathrm{L}|=O(1)$ and $\pr{W_\mathrm{L}=0}=\Omega(1)$.
\end{lemma}
\begin{proof}
   We assume that $I_{\mathrm{L}}$ is not
   empty (otherwise the lemma vacuously holds) and that (\ref{eq:truesuf}) is satisfied. 
   Note the importance of the assumption $c_{1}\leq r_{1}=o(N)$; if for instance $%
   r_{i}=\Omega (N)$ for all $i$, then all the entries from column $1$ could be contained in $I_{\mathrm{L}}$ and, therefore, 
   we would have $\pr{W_\mathrm{L}=0}=0$, violating the statement of the lemma.
   
   Let $\gamma =\sum _{(i,j)\in I_\mathrm{L}}r_ic_j/N$.
   We show that $\gamma,|I_\mathrm{L}|=O(1)$. For all sufficiently large $N$ we have
   \begin{equation}
   C
   \geq \sum_{(i,j)\in I_\mathrm{L}}\frac{r_i^{\underline{2}}c_j^{\underline{2}}}{2N^{\underline{2}}}
  \geq \sum_{(i,j)\in
   I_\mathrm{L}}\frac{r_ic_j\epsilon}{2N}
   = \frac{\gamma\epsilon}{2},  \label{eq:truesizel}
   \end{equation}
   where the second inequality is obtained from $(r_i-1)(c_j-1)\geq \epsilon N$,
   for all $(i,j)\in I_\mathrm{L}$ by the definition of $I_\mathrm{L}$. On the other hand, from the 
   definition of $I_\mathrm{L}$ we obtain 
   \begin{equation*}
   \gamma =\sum _{(i,j)\in I_\mathrm{L}}\frac{r_ic_j}{N}\geq \sum _{(i,j)\in I_\mathrm{L}%
   }\frac{(r_i-1)(c_j-1)}{N}\geq |I_\mathrm{L}|\epsilon.
   \end{equation*}
   Consequently, combining the previous estimate with (\ref{eq:truesizel}) we conclude that $\gamma,|I_\mathrm{L}|=O(1)$.
   
   It is useful to see Figure~\ref{fig:diag}(a,b) throughout this discussion. Let $%
   s_{i}=\sum_{j^{\prime }\colon (i,j^{\prime })\in I_{\mathrm{L}}}c_{j^{\prime
   }}$ and $t_{j}=\sum_{i^{\prime }\colon (i^{\prime },j)\in I_{\mathrm{L}%
   }}r_{i^{\prime }}$. In words, for any row $i\in I_{\mathrm{L}}^{\mathrm{r}}$%
   , $s_{i}$ is the sum of the column sums over all entries in $I_{\mathrm{L}}$
   corresponding to row $i$. Similarly, $t_{j}$ is defined for any column $j\in
   I_{\mathrm{L}}^{\mathrm{c}}$ as the sum of the row sums over all entries in $%
   I_{\mathrm{L}}$ corresponding to column $j$. Note that $\gamma N=\sum_{i\in
   I_{\mathrm{L}}^{\mathrm{r}}}r_{i}s_{i}=\sum_{j\in I_{\mathrm{L}}^{\mathrm{c}%
   }}c_{j}t_{j}$.
   
   We now derive a lower bound for $\pr{W_\mathrm{L}=0}$. Consider the row $1$,
   which belongs to $I_{\mathrm{L}}^{\mathrm{r}}$. The number of ways we can
   match the $r_{1}$ type-1 tokens of the first row with the $N-s_{1}$ type-2
   tokens available outside $I_{\mathrm{L}}$ is $(N-s_{1})^{\underline{r_{1}}}$%
   . Employing this reasoning for each row in $I_{\mathrm{L}}^{\mathrm{r}}$, we
   obtain (let $\ell=|I_{\mathrm{L}}^{\mathrm{r}}|$)%
   \begin{eqnarray}
   \pr{W_\mathrm{L}=0} &=&\frac{(N-s_{1})^{\underline{r_{1}}}}{N^{\underline{r_{1}}}}%
   \frac{(N-r_{1}-s_{2})^{\underline{r_{2}}}}{(N-r_{1})^{\underline{r_{2}}}}%
   \cdots \frac{(N-r_{1}-\cdots-r_{\ell-1}-s_{\ell})^{\underline{r_{\ell}}}}{%
   (N-r_{1}-\cdots-r_{\ell-1})^{\underline{r_{\ell}}}}  \notag \\
   &\geq &\left( 1-\frac{s_{1}}{N-\sum_{i^{\prime }\in I_{\mathrm{L}}^{\mathrm{r%
   }}}r_{i^{\prime }}}\right) ^{r_{1}}\cdots\left( 1-\frac{s_{\ell}}{%
   N-\sum_{i^{\prime }\in I_{\mathrm{L}}^{\mathrm{r}}}r_{i^{\prime }}}\right)
   ^{r_{\ell}}  \notag \\
   &=&\left( 1-\frac{s_{1}}{N-o(N)}\right) ^{r_{1}}\cdots\left( 1-\frac{s_{\ell}%
   }{N-o(N)}\right) ^{r_{\ell}}  \label{eq:truep2} \\
   &\geq &\func{exp}\left( -\sum_{i\in I_{\mathrm{L}}^{\mathrm{r}}}\frac{%
   r_{i}s_{i}}{[N-o(N)][1-O(s_{i}/N)]}\right)  \label{eq:truep1} \\
   &\geq &\func{exp}\left( -\gamma \right) -o(1),  \notag
   \end{eqnarray}%
   where in (\ref{eq:truep2}) we used the fact that $\sum_{i^{\prime }\in I_{%
   \mathrm{L}}^{\mathrm{r}}}r_{i^{\prime }}\leq r_{1}|I_{\mathrm{L}}^{\mathrm{r}%
   }|=o(N)$, and (\ref{eq:truep1}) comes from the fact that $(1-x)\geq \func{exp%
   }(-x/(1-x))$ for all $x\in \lbrack 0,1]$, which we apply with $%
   x=s_{i}/(N-o(N))$ for each $i\in I_{\mathrm{L}}^{\mathrm{r}}$. Moreover,
   since $r_{1}\geq c_{1}$ and (\ref{eq:truesuf}) holds, then $c_{j}=O(\sqrt{N}%
   ) $ for all $j$, and consequently, $s_{i}\leq c_{1}|I_{\mathrm{L}}^{\mathrm{c%
   }}|=O(\sqrt{N})$. This completes the proof of the lemma, since $\gamma = O(1)$.
\end{proof}
\medskip

For the term $\pr{Z_{\mathrm{S}}=0\mid W_{\mathrm{L}}=0}$, we use the fact that 
$I_{\mathrm{L}}$ contains only $O(1)$ entries to conclude that 
conditioning on $I_{\mathrm{L}}$ has a small effect. Then, we can carry
out the analysis as if no conditioning is being made, and we use the fact
that $(r_{i}-1)(c_{j}-1) < \epsilon N$ for $(i,j)\in I_{\mathrm{S}}$ to simplify the
calculations. 
The following lemma deals with this case. 

\begin{lemma}
\label{lem:truezs} 
If (\ref{eq:truesuf}) is
satisfied, 
then $\pr{Z_{\mathrm{S}}=0\mid W_{\mathrm{L}}=0}=\Omega(1)$.
\end{lemma}
The proof of Lemma~\ref{lem:truezs} is rather delicate and will require additional lemmas.
We then devote the remainder of this section to prove Lemma~\protect\ref{lem:truezs}.
Note that
the proof of Proposition~\ref{prop:suf} 
follows from Lemmas~\ref{lem:truezl} and~\ref{lem:truezs}.

\topic{Proof of Lemma~\protect\ref{lem:truezs}}
We will use 
the following quantity
$$
   \lambda = \lambda(N) = \sum_{(i,j)\in I_{\mathrm{S}}} \frac{r_{i}^{\underline {2}}c_{j}^{\underline{2}}}{2N^{2}}.
$$
Clearly, $\limsup_{N\to\infty}\lambda\leq C$ given (\ref{eq:truesuf}), and $\lambda = \E{Z_\mathrm{S}}$.

Before proceeding to the proof, we need to introduce some notation. For an
integer $k\geq 1$ and any $(i,j)\in I$, let ${\mathcal{B}}_k(i,j)$ be the
set of all sets of $k$ disjoint edges between row $i$ and column $j$, which
generalizes the definition of ${\mathcal{B}}_2(i,j)$ from Section~\ref%
{sec:sketch}. A typical element of ${\mathcal{B}}_k(i,j)$ has the form $%
\{e_1,e_2,\ldots ,e_k\}$, where $e_1,e_2,\ldots ,e_k$ are disjoint edges
between row $i$ and column $j$. Recall that the edges $e_1,e_2,\ldots ,e_k$
are said to be disjoint when each edge $e_{\ell }$, $1\leq \ell \leq k$,
corresponds to a pair of tokens $\{\beta _{\ell },\beta _{\ell }^{\prime }\}$
such that $\beta _1\neq \beta _2\neq \cdots \neq \beta _k$ and $\beta
_1^{\prime }\neq \beta _2^{\prime }\neq \cdots \neq \beta _k^{\prime }$,
that is, the edges $e_1,e_2,\ldots ,e_k$ do not share tokens. Intuitively, ${%
\mathcal{B}}_k(i,j)$ is the set of all possible ways we can match $k$ edges
between row $i$ and column $j$. So, clearly 
$|{\mathcal{B}}_k(i,j)|=r_i^{\underline k}c_j^{\underline k}/k!$, for all $(i,j)\in I$ and all $k$.

For any set of edges $B$, we define ${\mathcal{M}}(B)$ as the event that all
the edges in $B$ are matched by the configuration model. Note that the
occurrence of ${\mathcal{M}}(B)$ for $B\in {\mathcal{B}}_k(i,j)$ means that
the entry $T_{i,j}\geq k$. For all $k$, define ${\mathcal{B}}_k=\bigcup
_{(i,j)\in I_\mathrm{S}}{\mathcal{B}}_k(i,j)$. Finally, for all $k\geq 1$,
define the event ${\mathcal{P}}_k=\bigcup _{B\in {\mathcal{B}}_k}{\mathcal{M}%
}(B)$. In words, ${\mathcal{P}}_k$ is the event that there is an entry $%
T_{i,j}\geq k$ in $I_\mathrm{S}$. Using our notation we can write 
\begin{equation}
\pr{Z_\mathrm{S}\geq 1\mid W_\mathrm{L}=0}= \pr{{\mathcal{P}}_2\mid
W_\mathrm{L}=0}.
\label{eq:feq}
\end{equation}

When (\ref{eq:truesuf}) is satisfied, we can show that the probability that $%
{\mathcal{P}}_3$ occurs is small. We exploit this fact to simplify the
calculations via the following trivial inequalities 
\begin{equation}
\pr{{\mathcal{P}}_2\mid W_\mathrm{L}=0}\geq \pr{{\mathcal{P}}_2\cap 
{\mathcal{P}}_3^\mathrm{c}\mid W_\mathrm{L}=0} \label{eq:triviallb}
\end{equation}
and 
\begin{equation}
\pr{{\mathcal{P}}_2\mid W_\mathrm{L}=0}\leq \pr{{\mathcal{P}}_2\cap 
{\mathcal{P}}_3^\mathrm{c}\mid W_\mathrm{L}=0} + \pr{{\mathcal{P}}%
_3\mid W_\mathrm{L}=0}.  \label{eq:trivialub}
\end{equation}

We start the proof by stating a lemma that indicates that for any two
disjoint sets of disjoint edges $B$ and $B^{\prime }$, the probability that
the edges in $B$ are matched by the configuration model conditioning on 
$B^{\prime }$ being matched and $W_\mathrm{L}=0$ is essentially the same as without
conditioning. 

\begin{lemma}
\label{lem:truecond} Let $B$ and $B^{\prime }$ be disjoint sets of
disjoint edges from $I_{\mathrm{S}}$ such that the event ${\mathcal{M}}(B\cup B^{\prime })$ has non-zero
probability when conditioned on $W_\mathrm{L}=0$. Let $k$ and $k^{\prime }$ be the
number of edges in $B$ and $B^{\prime }$, respectively. If $k,k^{\prime
}=O(1)$, we obtain 
\begin{equation}
\pr{{\mathcal{M}}(B)\mid {\mathcal{M}}(B^{\prime }),W_\mathrm{L}=0}=\frac{1+o(1)}{(N-k^{\prime })^{\underline{k}}}.  \label{eq:truel1}
\end{equation}
\end{lemma}
\begin{proof}
   Before proceeding to the proof we show how to sample a table $T$ under the
   condition $W_\mathrm{L}=0$ and $\mathcal{M}(B^{\prime })$. If we were to condition only on $%
   \mathcal{M}(B^{\prime })$, it would suffice to disregard the tokens corresponding to the
   edges in $B^{\prime }$ and take a random pairing for the remaining tokens in
   a standard fashion. However, conditioning on $W_\mathrm{L}=0$ is more delicate. For
   example, conditioning on $W_\mathrm{L}=0$ implies that a pair of tokens from row $i$
   and column $j$ for which $(i,j)\in I_\mathrm{L}$ cannot be matched since all
   the entries in $I_\mathrm{L}$ are $0$.
   
   Define $s_i=\sum _{j^{\prime }\colon (i,j^{\prime })\in I_\mathrm{L}%
   }c_{j^{\prime }}$ and $t_j=\sum _{i^{\prime }\colon (i^{\prime },j)\in I_%
   \mathrm{L}}r_{i^{\prime }}$ (refer to Figure~\ref{fig:diag}(b) for a
   pictorial illustration of $s_i$ and $t_j$). Let $\rho _i$ be the number of
   edges in $B^{\prime }$ corresponding to row $i$, and let $\rho ^{\prime }_j$
   be the number of edges in $B^{\prime }$ corresponding to column $j$. We
   sample $T$ in a column-by-column manner, but the order according to which we
   sample the columns will matter. For each column $j$, there is a set $X_j$ of
   type-1 tokens that can be matched to type-2 tokens from column $j$ given $%
   W_\mathrm{L}=0 $ and $B^{\prime }$. Also, since the columns are in non-increasing
   order, we obtain that $X_j\subseteq X_{j^{\prime }}$ for all $j^{\prime }\in
   [j,n]$. Our strategy is to sample the columns in a non-increasing order,
   starting from column $1$ until column $|I_\mathrm{L}^\mathrm{c}|$. (Recall
   the definition $I_\mathrm{L}^\mathrm{c}=\{j\colon (i,j)\in I_\mathrm{L}\text{
   for some }i\}$.) Then, at that moment, all
   entries corresponding to a column in $I_\mathrm{L}$ have already been
   assigned. Therefore, sampling the remaining entries of the table does not
   dependent on $W_\mathrm{L}=0$ and we carry out the sampling trivially using the
   standard procedure for the configuration model.
   
   We now describe how to sample the entries from a column $j\in I_\mathrm{L}^%
   \mathrm{c}$ given that all columns from $1$ to $j-1$ have already been
   sampled. Note that there are $c_j-\rho ^{\prime }_j$ type-2 tokens still
   unmatched for column $j$. The main property we use is that $X_j\subseteq
   X_{j^{\prime }}$ for all $j^{\prime }\in [j,n]$, that is, the type-1 tokens
   that can be matched to the type-2 tokens from column $j$ can also be matched
   to any other column $j^{\prime }$ that have not yet been sampled. Therefore,
   it follows that each possible way to match the tokens from column $j$ is
   equally likely; we can take a uniformly random permutation of the type-1
   tokens in $X_j$ that have not been matched to any column $j'<j$ and select the first $c_j-\rho ^{\prime }_j$ to be
   matched to the type-2 tokens from column $j$.
   
   Recall that $k$ and $k^{\prime }$ are the cardinalities of $B$ and $%
   B^{\prime }$, respectively. Now we proceed to the proof of (\ref{eq:truel1}%
   ). First, assume $k=1$, i.e., $B=\{e_1\}=\{(\beta _1,\beta _1^{\prime })\}$,
   where $\beta _1$ is a type-1 token from some row $i$ and $\beta _1^{\prime }$
   is a type-2 token from some column $j$. We denote by ${\mathcal{I}}_j$ the
   set of entries in $I_\mathrm{L}$ corresponding to column $j^{\prime }\leq
   j-1 $. Formally, ${\mathcal{I}}_j=\{(i^{\prime },j^{\prime })\in I_\mathrm{L}%
   \colon j^{\prime }\in [1,j-1]\}$. We need to consider two cases.
   
   \par\medskip\noindent
   {\em Case 1:} $\beta _1^{\prime }$ corresponds to a column $j\in I_\mathrm{L%
   }^\mathrm{c}$.
   \par\noindent
   We write the probability that $\beta _1$ is matched to $\beta _1^{\prime }$
   as $q_1q_2$, where $q_1$ is the probability that $\beta _1$ has not been
   matched to any column $j^{\prime }\leq j-1$ and $q_2$ is the probability
   that $\beta _1$ is matched to $\beta _1^{\prime }$ given that it was not
   matched to any column $j^{\prime }\leq j-1$. We start with $q_2$. Note that
   there are $\zeta_j=N-t_j-\sum _{j^{\prime }=1}^{j-1}c_{j^{\prime }}
   -\left(k^{\prime }-\sum _{j^{\prime }=1}^{j-1}\rho ^{\prime
   }_{j^{\prime }}\right)$ type-1 tokens available to be matched to $\beta
   _1^{\prime }$. Note that $t_j \leq r_1 |I_\mathrm{L}|=o(N)$ and $\sum _{j^{\prime }=1}^{j-1}c_{j^{\prime }}-\sum
   _{j^{\prime }=1}^{j-1}\rho ^{\prime }_{j^{\prime }}=O(j\sqrt{N})$. 
   Therefore, since $j=O(1)$, we have
   $$
   q_2
   =\frac{1}{\zeta_j}
   =\frac{1}{N-k^{\prime }-o(N)}.
   $$
   For $q_1$, note first that the probability that $\beta_1$ is not matched to
   any type-2 token from a column $j'$ is
   $$
      \frac{(\zeta_{j'}-1)^{\underline {c_{j'}-\rho'_{j'}}}}{\zeta_{j'}^{\underline {c_{j'}-\rho'_{j'}}}}
      = \frac{\zeta_{j'}-c_{j'}+\rho'_{j'}}{\zeta_{j'}}
      = 1 - O(1/\sqrt{N}),
   $$
   since $k'=O(1)$. Recall that $i$ is the row associated with the token $\beta_1$. Clearly, when assigning a column $j'$ for which 
   $(i,j')\in I_\mathrm{L}$, we have that $\beta_1$ will not be matched to any token from column $j'$ since we are 
   conditioning on $W_\mathrm{L}=0$. Therefore, we obtain for $q_1$
   \begin{equation*}
   q_1=\prod _{j'\leq j-1\colon (i,j')\not\in I_\mathrm{L}}\left(1 - O(1/\sqrt{N})\right)
   =1-O(1/\sqrt{N}),
   \end{equation*}
   since $j=O(1)$. We then obtain that $\beta _1$ is
   matched to $\beta _1^{\prime }$ with probability $\frac{1}{N-k^{\prime }}%
   (1+o(1))$.

   \par\medskip\noindent
   {\em Case 2:}  $\beta _1^{\prime }$ corresponds to a column $j\not \in I_%
   \mathrm{L}^\mathrm{c}$.
   \par\noindent
   Again, with probability $1-O(1/\sqrt{N})$, $\beta _1$ was not matched
   to any type-2 token from columns in $I_\mathrm{L}^\mathrm{c}$. When this
   happens, there are still $N-\sum _{j^{\prime }\in I_\mathrm{L}^\mathrm{c}%
   }c_{j^{\prime }}-\left(k^{\prime }-\sum _{j^{\prime }\in I_\mathrm{L}^%
   \mathrm{c}}\rho ^{\prime }_{j^{\prime }}\right)=N-k^{\prime }-O(\sqrt{N})$
   type-1 tokens to be matched to $\beta _1^{\prime }$ and the probability that 
   $\beta _1$ and $\beta _1^{\prime }$ are matched is $\frac{1}{N-k^{\prime }}%
   (1+O(1/\sqrt{N}))$.
   
   \medskip
   Therefore, for $k=1$ and $k^{\prime }=O(1)$, we obtain $\pr{{\mathcal{M%
   }}(B)\mid {\mathcal{M}}(B^{\prime }),W_\mathrm{L}=0 }=\frac{1}{N-k^{\prime }}(1+o(1))$. When $k\geq 2$, let $B=\{e_1,e_2,\ldots ,e_k\}$ and $B_{\ell }$
   be the first $\ell $ edges in $B$, i.e., $B_{\ell }=\{e_1,e_2,\ldots
   ,e_{\ell }\}$. For convenience, define $B_0=\emptyset $. Therefore 
   \begin{eqnarray*}
   \pr{{\mathcal{M}}(B)\mid {\mathcal{M}}(B^{\prime }),W_\mathrm{L}=0 }&=&\prod
   _{\ell =1}^{k} \pr{{\mathcal{M}}(\{e_{\ell }\})\mid {\mathcal{M}}%
   (B_{\ell -1}),{\mathcal{M}}(B^{\prime }),W_\mathrm{L}=0} \\
   &=&\prod _{\ell =1}^{k}\left (\frac{1}{N-k^{\prime }-\ell +1}(1+o(1))\right ),
   \end{eqnarray*}
   which concludes the proof of Lemma~\ref{lem:truecond} since $k,k^{\prime
   }=O(1)$.
\end{proof}
\medskip

To simplify the calculations to follow,
we first solve the simpler case when $\lambda=o(1)$. For this we prove Lemma~\ref{lem:smalllambda} below, which gives a stronger 
result.
\begin{lemma}
   If $\limsup_{N\to\infty}\lambda<1$, we have 
   $\pr{Z_\mathrm{S}=0 \mid W_\mathrm{L}=0}=\Omega(1)$.
\label{lem:smalllambda}
\end{lemma}
\begin{proof}
   We will show that 
   $\liminf_{N\to\infty}\pr{Z_\mathrm{S}=0 \mid W_\mathrm{L}=0}>0$. 
   Using Markov's inequality, we can write
   $$
      \liminf_{N\to\infty}\pr{Z_\mathrm{S}=0 \mid W_\mathrm{L}=0}
      = 1-\limsup_{N\to\infty}\pr{Z_\mathrm{S}\geq 1 \mid W_\mathrm{L}=0}
      \geq 1-\limsup_{N\to\infty}\E{Z_\mathrm{S} \mid W_\mathrm{L}=0}.
   $$
   Using linearity of expectation, we obtain
   $\E{Z_\mathrm{S} \mid W_\mathrm{L}=0} = \sum_{(i,j)\in I_\mathrm{S}}\sum_{B\in\mathcal{B}_2(i,j)} 
   \pr{\mathcal{M}(B) \mid W_\mathrm{L}=0}$. From Lemma~\ref{lem:truecond} we have 
   $\pr{\mathcal{M}(B) \mid W_\mathrm{L}=0} = (1+o(1))/N^{\underline 2}$, 
   which gives
   $$
      \liminf_{N\to\infty}\pr{Z_\mathrm{S}=0 \mid W_\mathrm{L}=0}
      \geq 1 - \limsup_{N\to\infty} \sum_{(i,j)\in I_\mathrm{S}} 
      \frac{r_i^{\underline{2}}c_j^{\underline{2}}}{2N^{\underline{2}}}(1+o(1))
      \geq 1-\limsup_{N\to\infty}\lambda.
   $$
\end{proof}
\medskip

From now on, we will assume that $\lambda=\Theta(1)$. The case where $\lambda$ is neither $o(1)$ nor
$\Omega(1)$ can be handled by Lemma~\ref{lem:applyssseqp} by setting $f_N(\mathbf{r},\mathbf{c})=\lambda$ and 
$\mathcal{Z}$ as the space of sequences satisfying the conditions in Proposition~\ref{prop:suf}.

Now we use Lemma~\ref%
{lem:truecond} to show that the bounds in (\ref{eq:triviallb}) and (\ref%
{eq:trivialub}) are tight up to smaller-order terms. This simplification is
the main reason for treating $I_{\mathrm{L}}$ and $I_{\mathrm{S}}$
separately.

\begin{lemma}
\label{lem:truep3} Conditional on $W_\mathrm{L}=0$, the
probability that the configuration model creates three edges for any entry
in $I_\mathrm{S}$ can be upper bounded by $\pr{{\mathcal{P}}_3\mid
W_\mathrm{L}=0} \leq \lambda \epsilon/3+o(1)$.
\end{lemma}
\begin{proof}
   For any $%
   (i,j)\in I_{\mathrm{S}}$, the number
   of ways to match three edges from $(i,j)$ is $r_{i}^{\underline{3}%
   }c_{j}^{\underline{3}}/3!$, and 
   $$
      \pr{{\mathcal{P}}_{3}\mid W_\mathrm{L}=0} \leq \sum_{(i,j)\in I_{\mathrm{S}%
      }}\sum_{B\in {\mathcal{B}}_{3}(i,j)}\pr{{\mathcal{M}}(B)\mid W_\mathrm{L}=0}.
   $$
   We use $(r_{i}-1)(c_{j}-1) < \epsilon N$
   and Lemma~\ref{lem:truecond} to conclude that for all $(i,j)\in I_{\mathrm{S}%
   }$%
   $$
   \pr{{\mathcal{P}}_{3}\mid W_\mathrm{L}=0}
   \leq\sum_{(i,j)\in I_{\mathrm{S}}}\frac{r_{i}^{\underline{3}}c_{j}^{%
   \underline{3}}}{3!}\frac{(1+o(1))}{N^{\underline{3}}}
   < \sum_{(i,j)\in I_{\mathrm{S}}}\frac{r_{i}^{\underline{2}}c_{j}^{\underline{2}}}{6N^{\underline{2}}}\epsilon(1+o(1)),
   $$
   which together with (\ref{eq:truesuf}) yields the validity of Lemma \ref{lem:truep3}.
\end{proof}
\medskip

It remains to derive a bound for the term $\pr{{\mathcal{P}}%
_2\cap {\mathcal{P}}_3^\mathrm{c}\mid W_\mathrm{L}=0}$. We apply the
inclusion-exclusion principle. Define the set ${\mathcal{D}}_{\ell }$ to
contain all possible sets of the form $\{d_1,d_2,\ldots ,d_{\ell }\}$ with $%
d_1,d_2,\ldots ,d_{\ell }$ being distinct double edges (i.e., distinct
elements from ${\mathcal{B}}_2$). Ideally, we would like $(d_1,d_2,\ldots
,d_{\ell })\in {\mathcal{D}}_{\ell }$ to represent a possible choice of $%
\ell $ double edges for $\ell $ distinct entries of the table. However, it
may be the case that, say, $d_1$ and $d_2$ are double edges for the same
entry. In this case, if $d_1$ and $d_2$ have one edge in common, then they
correspond to having $3$ edges matched for that entry. Otherwise, if $d_1$
and $d_2$ do not have an edge in common, then they correspond to having $4$
edges matched for that entry. In any of these cases, we count these terms in
the event ${\mathcal{P}}_3$, which we treat separately using Lemma~\ref%
{lem:truep3}. That is the reason why we derive the probability for ${%
\mathcal{P}}_2\cap {\mathcal{P}}_3^\mathrm{c}$ instead of working directly
with ${\mathcal{P}}_2$. This is also the main benefit we obtain from considering the 
entries in $I_\mathrm{L}$ and $I_\mathrm{S}$ separately. 
The elements of ${\mathcal{D}}_{\ell }$ that count
for the event ${\mathcal{P}}_2\cap {\mathcal{P}}_3^\mathrm{c}$ are only
those corresponding to $\ell $ double edges from $\ell $ distinct entries.
There is also one additional case. There exist terms of the form $%
D=\{d_1,d_2,\ldots ,d_{\ell }\}\in {\mathcal{D}}_{\ell }$ such that, say, $%
d_1$ and $d_2$ share a token. This happens if for example $d_1$ is a double
edge for the entry $(i,j)$, $d_2$ is a double edge for the entry $%
(i,j^{\prime })$, and one of the type-1 tokens from row $i$ contained in $%
d_1 $ is also contained in $d_2$. However, should that be the case, then $%
d_1 $ and $d_2$ cannot occur simultaneously and the event that all double
edges in $D$ are matched by the configuration model has probability $0$.

Abusing notation slightly, for an element $D\in {\mathcal{D}}_{\ell }$, we
denote by ${\mathcal{M}}(D)$ the event that all $\ell $ double edges in $D$
are matched by the configuration model. Therefore, using the
inclusion-exclusion principle, we obtain 
\begin{equation}
   \pr{{\mathcal{P}}%
   _2\cap {\mathcal{P}}_3^\mathrm{c}\mid W_\mathrm{L}=0}=\sum _{\ell \geq
   1}(-1)^{\ell +1}p_{\ell },
   \label{eq:part1}
\end{equation} where $p_{\ell }=\sum _{D\in {\mathcal{D}}_{\ell
}} \pr{{\mathcal{M}}(D)\cap {\mathcal{P}}_3^\mathrm{c}\mid W_\mathrm{L}=0}$. We will 
take a value $L>0$ sufficiently large that we will set later and consider the value of $p_\ell$ for 
$\ell\leq L$. The following lemma gives lower and upper bounds for $p_\ell$.

\begin{lemma}
\label{lem:pell}
Assume $\lambda=\Omega(1)$ and fix $L$.
Let $\xi>0$ be an arbitrarily small constant. 
We can set $\epsilon=\epsilon(\lambda,L,\xi)$ sufficiently small in the definition of $I_\mathrm{L}$ so that 
for all $\ell\leq L$ we have
\begin{equation}
   \frac{\lambda^\ell}{\ell!}(1-\xi-o(1))
   \leq p_{\ell } 
   \leq \frac{\lambda^\ell}{\ell!}+o(1)
\label{eq:truepell}
\end{equation}
\end{lemma}
\begin{proof}
   Let ${\mathcal{J}}_{\ell }$ be the set of all sets of $\ell $ distinct
   elements from $I_{\mathrm{S}}$, that is, $J\in {\mathcal{J}}_{\ell }$ has
   the form $J=\{(i_{1},j_{2}),(i_{2},j_{2}),\ldots ,(i_{\ell },j_{\ell })\}$,
   where each $(i_{k},j_{k})$ corresponds to a distinct entry from $I_{\mathrm{S%
   }}$. Now let ${\mathcal{B}}_{2}(J)$, for $J\in {\mathcal{J}}_{\ell }$, be
   the set of all possible ways we can choose one double edge from each element
   of $J$. That is, if $J=\{(i_{1},j_{2}),(i_{2},j_{2}),\ldots ,(i_{\ell
   },j_{\ell })\}$, then a typical element from $B\in {\mathcal{B}}_{2}(J)$ has
   the form $B=\{d_{1},d_{2},\ldots ,d_{\ell }\}$, where $d_{k}$ is a double
   edge corresponding to the entry $(i_{k},j_{k})$, for $1\leq k\leq \ell $.
   Recall that in the summation in (\ref{eq:truepell}), we obtain $\pr{{%
   \mathcal{M}}(D)\cap {\mathcal{P}}^\mathrm{c}_{3}\mid W_\mathrm{L}=0}=0$ for all $D\in 
   {\mathcal{D}}_{\ell }$ that contain two or more distinct double edges that
   share a token. Therefore, we obtain the following equality 
   \begin{equation}
   p_{\ell }=\sum_{D\in {\mathcal{D}}_{\ell }}\pr{{\mathcal{M}}(D)\cap {%
   \mathcal{P}}_{3}^{\mathrm{c}}\mid W_\mathrm{L}=0}=\sum_{J\in {\mathcal{J}}_{\ell
   }}\sum_{B\in {\mathcal{B}}_{2}(J)}\pr{{\mathcal{M}}(B)\cap {\mathcal{P}%
   }_{3}^{\mathrm{c}}\mid W_\mathrm{L}=0},  \label{eq:pl1}
   \end{equation}%
   where the last term translates to 
   \begin{equation}
   \pr{{\mathcal{M}}(B)\cap {\mathcal{P}}_{3}^{\mathrm{c}}\mid W_\mathrm{L}=0 }=%
   \pr{{\mathcal{P}}_{3}^{\mathrm{c}}\mid {\mathcal{M}}(B),W_\mathrm{L}=0}\pr{{\mathcal{M}}(B)\mid W_\mathrm{L}=0}.  
   \label{eq:pl2}
   \end{equation}
   
   We start with the term $\pr{{\mathcal{P}}_{3}^{\mathrm{c}}\mid {%
   \mathcal{M}}(B),W_\mathrm{L}=0 }=1-\pr{{\mathcal{P}}_{3}\mid {\mathcal{M}}%
   (B),W_\mathrm{L}=0 }$. For each $(i,j)\in J$, $B$ contains at least one double edge for the entry 
   $(i,j)$. Therefore, the probability that ${\mathcal{P}}_{3}$ happens can be
   upper bounded by the probability that another edge corresponding to some
   entry in $J$ is matched plus the probability that $3$ edges for some entry
   not in $J$ are matched, that is, $\pr{{\mathcal{P}}_{3}\mid {\mathcal{M}}(B),W_\mathrm{L}=0}$
   is at most 
   $$
      \sum_{(i,j)\in
      J}\sum_{B^{\prime }\in B_{1}(i,j)\setminus \mathbf{\cup }B}\pr{{%
      \mathcal{M}}(B^{\prime })\mid {\mathcal{M}}(B),W_\mathrm{L}=0 }+\sum_{(i,j)\not\in
      J}\sum_{B^{\prime }\in B_{3}(i,j)}\pr{{\mathcal{M}}(B^{\prime })\mid {%
      \mathcal{M}}(B),W_\mathrm{L}=0 },
   $$
   where the notation $\mathbf{\cup }B$ represents the union of all elements in 
   $B=\{d_{1},d_{2},\ldots ,d_{\ell }\}$ (recall that the $d_{i}$'s are double
   edges). Note that $\mathbf{\cup }B$ contains $2\ell $ edges. Therefore, for $%
   \ell =O(1)$ we can use Lemma~\ref{lem:truecond} for the first term and a
   derivation similar to the proof of Lemma~\ref{lem:truep3} for the second
   term to obtain 
   \begin{eqnarray}
   \lefteqn{\pr{{\mathcal{P}}_{3}\mid {\mathcal{M}}(B),W_\mathrm{L}=0}}\nonumber\\
    &\leq&\sum_{(i,j)\in J}\frac{(r_{i}-2)(c_{j}-2)}{N-2\ell }(1+o(1))+\sum_{(i,j)\not\in J}\frac{r_{i}^{\underline{3}}c_{j}^{\underline{3}}}{%
   3!(N-2\ell )^{\underline{3}}}(1+o(1))\\
   &\leq& \epsilon(L+\lambda/3) + o(1),  \label{eq:truep3}
   \end{eqnarray}%
   uniformly in $J$ and $B$.
   
   Now we turn to the term $\sum _{J\in {\mathcal{J}}_{\ell }}\sum _{B\in {%
   \mathcal{B}}_2(J)} \pr{{\mathcal{M}}(B)\mid W_\mathrm{L}=0}$, which corresponds to an upper bound
   for the right hand side of (\ref{eq:pl1}). Our goal is to
   write this term recursively in $\ell $. First, notice that the case $\ell =1$
   reduces to 
   \begin{equation*}
   \sum _{J\in {\mathcal{J}}_1}\sum _{B\in {\mathcal{B}}_2(J)} \pr{{%
   \mathcal{M}}(B)\mid W_\mathrm{L}=0}=\sum _{(i,j)\in I_\mathrm{S}}\frac{%
   r_i^{\underline 2}c_j^{\underline 2}}{2N^{\underline 2}}(1+o(1))=\lambda +o(1),
   \end{equation*}
   where $\lambda =\sum _{(i,j)\in I_\mathrm{S}}r_i^{\underline
   2}c_j^{\underline 2}/(2N^{\underline 2})$ as defined in Lemma~\ref%
   {lem:truezs}.
   
   Now, consider a fixed $J\in {\mathcal{J}}_{\ell }$. Note that we can write $%
   J=J^{\prime }\cup (i,j)$ where $J^{\prime }\in {\mathcal{J}}_{\ell -1}$ and $%
   (i,j)\in I_\mathrm{S}\setminus J^{\prime }$. Note also that there are $\ell $
   possible such values for $J^{\prime }\subset J$. For a fixed $J'$, 
   let $\eta_i=\eta _i(J^{\prime })$
   be the number of elements in $J^{\prime }$ corresponding to an entry in row $%
   i$, that is, $\eta _i=|\{(k,k^{\prime })\in J^{\prime }\colon
   k=i\}|$. Likewise, let $\eta_j'=\eta _j^{\prime }(J^{\prime })$ be the number of
   elements in $J^{\prime }$ corresponding to an entry in column $j$.
   Therefore, given $B^{\prime }\in {\mathcal{B}}_2(J^{\prime })$, 
   the sum of row $i$ becomes $r_i-2\eta _i$ 
   and the sum of column $j$ becomes $%
   c_j-2\eta_j^{\prime }$.
   For $\ell =O(1)$, we can apply Lemma~\ref{lem:truecond} to derive the
   following equality 
   \begin{eqnarray}
   \lefteqn {\sum _{J\in {\mathcal{J}}_{\ell }}\sum _{B\in {\mathcal{B}}_2(J)} 
   \operatorname {Pr}[{\mathcal{M}}(B)\mid W_\mathrm{L}=0 ]} \nonumber\\
   &=&\frac{1}{\ell }\sum _{J^{\prime }\in {\mathcal{J}}_{\ell -1}}\sum
   _{B^{\prime }\in {\mathcal{B}}_2(J^{\prime })} \pr{{\mathcal{M}}%
   (B^{\prime })\mid W_\mathrm{L}=0}\sum _{(i,j)\in I_\mathrm{S}\setminus J^{\prime
   }}\sum _{B^{\prime \prime }\in {\mathcal{B}}_2(i,j)} \pr{{\mathcal{M}}%
   (B^{\prime \prime })\mid {\mathcal{M}}(B^{\prime }),W_\mathrm{L}=0} \nonumber\\
   &=&\frac{1}{\ell }\sum _{J^{\prime }\in {\mathcal{J}}_{\ell -1}}\sum
   _{B^{\prime }\in {\mathcal{B}}_2(J^{\prime })} \pr{{\mathcal{M}}%
   (B^{\prime })\mid W_\mathrm{L}=0}(1+o(1)) 
    \sum _{(i,j)\in I_\mathrm{S}\setminus J^{\prime }}\frac{(r_i-2\eta
    _i)^{\underline
    2}(c_j-2\eta ^{\prime }_j)^{\underline 2}}{2(N-2\ell +2)^{\underline 2}}.
   \label{eq:recursion}
   \end{eqnarray}
   Note that only pairs $(i,j)$ with $r_i,c_j \geq 2$ count for the last sum in (\ref{eq:recursion}). So in what follow we 
   assume that $r_i,c_j\geq 2$.
   Note that $\sum
   _i\eta _i=\sum _{j}\eta _j^{\prime }=\ell
   -1$, and letting
   \begin{equation}
      X=\sum_{(i,j)\in I_\mathrm{S}\setminus J'}\frac{r_i^{\underline{2}}c_j^{\underline{2}}}{2(N-2\ell+2)^{\underline{2}}}
      -\sum_{(i,j)\in I_\mathrm{S}\setminus J'}\frac{(r_i-2\eta_i)^{\underline{2}}(c_j-2\eta_j')^{\underline{2}}}{2(N-2\ell+2)^{\underline{2}}},
      \label{eq:defx}
   \end{equation}
   we have
   $$
      \sum _{(i,j)\in I_\mathrm{S}\setminus J'}\frac{(r_i-2\eta
         _i)^{\underline 2}(c_j-2\eta_j')^{\underline 2}}{2(N-2\ell
         +2)^{\underline 2}}
      =\sum _{(i,j)\in I_\mathrm{S}}\frac{r_i^{\underline 2}c_j^{\underline 2}}{%
      2(N-2\ell +2)^{\underline 2}}-\sum _{(i,j)\in J^{\prime }}\frac{%
      r_i^{\underline 2}c_j^{\underline 2}}{2(N-2\ell +2)^{\underline 2}} -X.
   $$
   If we apply the condition $(r_i-1)(c_j-1)< \epsilon N$ and use the inequality $y \leq 2(y-1)$ valid for all $y \geq 2$, 
   we can write the second term in the right hand side above as 
   $$
      \sum _{(i,j)\in J^{\prime }}\frac{%
      r_i^{\underline 2}c_j^{\underline 2}}{2(N-2\ell +2)^{\underline 2}}
      \leq \sum _{(i,j)\in J^{\prime }}\frac{}{}
      \frac{2(r_i-1)^2(c_j-1)^2}{N^2}(1+O(1/N))
      \leq 2\epsilon^2(\ell-1) + O(1/N).
   $$
   If we expand $X$ in (\ref{eq:defx}) we obtain $X=X_1-X_2$, where 
   $$
      X_1 = \sum_{(i,j)\in I_\mathrm{S}\setminus J'}
         \frac{2 \eta_j'(2 c_j-1)r_i^{\underline{2}} + 2 \eta_i (2 r_i-1) c_j^{\underline{2}} + 8 \eta_i {\eta_j'}^2 (2 r_i-1) + 8 \eta_i^2 \eta_j' (2 c_j-1)}{2(N-2\ell+2)^{\underline{2}}}
   $$
   and
   $$
      X_2 = \sum_{(i,j)\in I_\mathrm{S}\setminus J'}
         \frac{4 {\eta_j'}^2 r_i^{\underline{2}} + 4 \eta_i^2c_j^{\underline{2}} + 4\eta_i\eta_j'(2 r_i -1)(2c_j-1) + 16\eta_i^2{\eta_j'}^2}{2(N-2\ell+2)^{\underline{2}}}.
   $$
   Since $r_i,c_j\geq 2$, we have that $(r_i-1)(2 c_j-1)$ and $(2r_i-1)(c_j-1)$ can both be upper bounded by $3(r_i-1)(c_j-1)$.
   Then, applying $(r_i-1)(c_j-1) < \epsilon N$ to the first two terms of $X_1$ and using the fact that 
   $\sum_{j=1}^n \eta_i{\eta'_j}^2 \leq \ell^2\sum_{j=1}^n \eta'_j\leq \ell^3$ and similarly 
   $\sum_{i=1}^m \eta_i^2\eta'_j \leq \ell^2\sum_{i=1}^m \eta_i\leq \ell^3$ for the last two terms of $X_1$, 
   we have $X_1\leq 6\epsilon \ell + O(1/N)$, and using the simple fact $X_2 \geq 0$ we get
   \begin{equation}
      \sum_{(i,j)\in I_\mathrm{S}\setminus J'}\frac{(r_i-2\eta
         _i)^{\underline 2}(c_j-2\eta _j')^{\underline 2}}{2(N-2\ell
         +2)^{\underline 2}}
      \geq \lambda - 2\epsilon^2\ell - 6\epsilon \ell - O(1/N).
      \label{eq:term1}
   \end{equation}
   Putting~(\ref{eq:recursion}) and~(\ref{eq:term1}) together, and iterating this procedure 
   $\ell$ times for $\ell =O(1)$ we obtain 
   \begin{equation}
      \frac{(\lambda - 2\epsilon^2L -6\epsilon L)^\ell}{\ell!}-o(1)
      \leq 
      \sum _{J\in {\mathcal{J}}_{\ell }}\sum _{B\in {\mathcal{B}}_2(J)} \pr{{%
      \mathcal{M}}(B)\mid W_\mathrm{L}=0}
      \leq \frac{\lambda ^{\ell }}{\ell!} + o(1).
   \label{eq:trueb}
   \end{equation}
   Since $\lambda=\Omega(1)$ and $\epsilon$ is sufficiently small, we can find a constant $c$ such that 
   $$
      \frac{(\lambda - 2\epsilon^2L -6\epsilon L)^\ell}{\ell!}
      \geq \frac{\lambda^\ell}{\ell!}(1-c\epsilon L \ell)
      \geq \frac{\lambda^\ell}{\ell!}(1-c\epsilon L^2).
   $$
   
   Putting (\ref{eq:pl2}), (\ref{eq:truep3}), and (\ref{eq:trueb}) together,
   and plugging the result into (\ref{eq:pl1}), we obtain 
   $$
      \frac{\lambda^\ell}{\ell!}(1-c\epsilon L^2-\epsilon(L+\lambda/3)-o(1))
      \leq p_{\ell } 
      \leq \frac{\lambda^\ell}{\ell!}+o(1).
   $$
   This concludes the proof of the
   lemma since we can set $\epsilon$ sufficiently small so that $\xi \leq \epsilon(cL^2+L+\lambda/3)$.
\end{proof}
\medskip

For some fixed constant $L$, we can use Bonferroni inequality to obtain a lower bound for $1- \pr{{%
\mathcal{P}}_2\cap {\mathcal{P}}_3^\mathrm{c}\mid W_\mathrm{L}=0}=1-\sum _{\ell \geq
1}(-1)^{\ell +1}p_{\ell }$ via 
\begin{equation*}
   1-\sum _{\ell =1}^L(-1)^{\ell +1}p_{\ell }
   \geq 1-\sum_{\ell=1}^L(-1)^{\ell+1}\frac{\lambda^\ell}{\ell!}-\sum_{\ell=1}^L\frac{\lambda^\ell\xi}{\ell!}-o(1)
   \geq e^{-\lambda }-\lambda^L/L!-e^{\lambda}\xi-o(1),
\end{equation*}
where $\xi$ is obtained from Lemma~\ref{lem:pell}.

Recall that (\ref{eq:truesuf}) implies $\lambda =O(1)$. For an arbitrarily
small constant $\delta >0$ independent of $N$ (as long as $N$ is
sufficiently large) we can set $L$ large enough so that 
$\lambda^L/L!\leq \delta/4$. 
We can also set $\xi$ in Lemma~\ref{lem:pell} so that $e^{\lambda}\xi\leq\delta/4$. 
Now, having fixed $L$ and $\xi$, we can set $\epsilon$ small enough so that 
Lemma~\ref{lem:pell} can be applied and in addition we have $\lambda\epsilon/3\leq \delta/4$. 
Then we put together (\ref{eq:feq}) and (\ref{eq:trivialub}), and use Lemma~\ref{lem:truep3}
and (\ref{eq:part1}) to 
obtain for large enough $N$ 
\begin{eqnarray*}
\pr{Z_{\mathrm{S}}=0\mid W_{\mathrm{L}%
}=0}
&\geq& 1-\sum _{\ell =1}^L(-1)^{\ell +1}p_{\ell } - \lambda\epsilon/3-o(1)\\
&\geq& e^{-\lambda }-\frac{3\delta}{4}-o(1)\\
&\geq& e^{-\lambda }-\delta,
\end{eqnarray*}%
which concludes the proof of Lemma~\ref{lem:truezs}.
Then, Proposition~\ref{prop:suf} follows immediately from Lemmas~\ref{lem:truezl} and~\ref{lem:truezs}.

\section{Proof of Proposition~\protect\ref{prop:extra}}
\label{sec:extra}

The derivations in this section require a careful analysis of subsequences. For this reason, in this section, 
we will use the full notation for input sequences as described in the paragraph of~\eqref{eq:subsequence}.
Let $S=(\mathbf{r}(\ell),\mathbf{c}(\ell),N(\ell))_{\ell}$ be an input sequence. Recall the definition 
$$
   \kappa(S) =\min \{i\geq 1\colon r_{i}(\ell)=o(N(\ell)) \text{ as $\ell\to\infty$}\}
$$ 
from~\eqref{eq:kappa} 
and note that
\begin{equation}
   \kappa(S) \geq \kappa(S') \text{ for any subsequence $S'=(\mathbf{r}'(\ell'),\mathbf{c}'(\ell'),N'(\ell'))_{\ell'}$ of $S$}.
   \label{eq:kappamonotone}
\end{equation}
Now let 
$$
   \kappa'(S)=\min \Big\{j\geq 1\colon \limsup_{\ell\to \infty }c_{j}(\ell)\leq 1\Big\}.
$$
We define $\kappa (S)$ and $\kappa'(S)$ since it suffices to look at the entries in $[1,\kappa(S)-1]\times [ 1,\kappa'(S)-1]$.

Throughout this section, we assume that, as $\ell\to\infty$, 
we have 
\begin{equation}
   \sum_{(i,j)\in I}r_{i}^{\underline{2}}(\ell)c_{j}^{\underline{2}}(\ell)=O(N^{2}(\ell)) 
   \quad\text{ and }\quad \liminf_{\ell\to\infty}c_{1}(\ell)\geq 2 
   \quad\text{ and }\quad r_{1}(\ell)=\Omega (N(\ell)).
   \label{eq:condsec}
\end{equation}
These conditions immediately imply that 
\begin{equation}
   \kappa(S) ,\kappa'(S)\geq 2 
   \quad\text{ and }\quad
   c_{1}(\ell),\kappa'(S)=O(1) \text{ as $\ell\to\infty$}.
   \label{eq:usefulfacts}
\end{equation}


First we show in Lemma~\ref{lem:extranec} that if condition~\ref{it:n1} of Theorem~\ref{thm:main} is true, then 
condition~\ref{it:n2} is necessary.
\begin{lemma}
\label{lem:extranec} For any sequence $S=(\mathbf{r}(\ell),\mathbf{c}(\ell),N(\ell))_{\ell}$ satisfying the conditions in~\eqref{eq:condsec}, if 
there exists a subsequence $S'=(\mathbf{r}'(\ell'),\mathbf{c}'(\ell'),N'(\ell'))_{\ell'}$ of 
$S$ for which $\sum_{i=\kappa(S')}^\infty r'_i(\ell')=o(N'(\ell'))$ as $\ell'\to\infty$
and $\limsup_{\ell'\to \infty} c'_1(\ell) \geq \kappa(S')$,
then $\liminf_{\ell\to\infty}\pr{T\in \Omega _{\mathbf{r},\mathbf{c}}}=0$.
\end{lemma}
\begin{proof}
   Since $c_1(\ell)=O(1)$ as $\ell\to\infty$ and $\limsup_{\ell'\to \infty} c'_1(\ell) \geq \kappa(S')$, we have $\kappa(S')<\infty$.
   Also, there exists a subsequence $S''=({\bf r}''(\ell''),{\bf c}''(\ell''),N''(\ell''))_{\ell''}$ of $S'$ having the following two properties:
   \begin{enumerate}[(i)]
      \item $\sum_{i=\kappa(S'')}^\infty r''_i(\ell'')=o(N''(\ell''))$ as $\ell''\to\infty$.
      \item $\lim_{\ell''\to \infty} c''_1(\ell'') \geq \kappa(S')\geq \kappa(S'')$. 
   \end{enumerate}
   Property~(ii) follows immediately by the definition of $\limsup$ and the monotonicity of $\kappa$ from~\eqref{eq:kappamonotone}. 
   Now note that, to establish property~(i), we have that, as $\ell''\to\infty$,
   \begin{align*}
      \sum_{i=\kappa(S'')}^\infty r''_i(\ell'')
      &= \sum_{i=\kappa(S'')}^{\kappa(S')-1} r''_i(\ell'') + \sum_{i=\kappa(S')}^\infty r''_i(\ell'')\\
      &\leq r''_{\kappa(S'')}(\ell'') (\kappa(S')-\kappa(S'')) + o(N''(\ell''))
      =o(N''(\ell'')),
   \end{align*}
   where the inequality uses the fact that $r''_i(\ell'')$ is non-increasing with $i$ for any fixed $\ell''$ and that $\sum_{i=\kappa(S')}^\infty r'_i(\ell')=o(N'(\ell'))$ as $\ell'\to\infty$.
   Then, the last step follows since $\kappa(S')<\infty$ and $r_i''(\ell'')=o(N''(\ell''))$ as $\ell''\to\infty$ for any $i \geq \kappa(S'')$ by the definition of $\kappa$.
   We will now show that 
   $$
      \lim_{\ell''\to\infty}\pr{T\in \Omega _{\mathbf{r}'',\mathbf{c}''}}=0,
   $$
   which clearly implies that $\liminf_{N\to\infty}\pr{T\in \Omega _{\mathbf{r},\mathbf{c}}}=0$.
   Using the fact that $r_1(\ell)=\Omega(N(\ell))$, which gives that $r''_1(\ell'')=\Omega(N''(\ell''))$, we have 
   \begin{equation}
      O(1)= \frac{1}{r''_1(\ell'')(r''_1(\ell'')-1)}\sum _{(i,j)\in I}{r_i''(\ell'')}^{\underline 2}{c_j''(\ell'')}^{\underline 2}
      \geq \sum_{j=1}^{\kappa'(S'')-1} c''_j(\ell'')(c''_j(\ell'')-1).
      \label{eq:sumc}
   \end{equation}
   Then, we have the following result for the probability that all 
   entries $(i,j)$ with $i\geq\kappa(S'')$ and $j\leq \kappa'(S'')-1$ are $0$:
   \begin{eqnarray}
   \pr{\bigcap _{i=\kappa(S'')}^\infty\bigcap_{j=1}^{\kappa'(S'')-1}
   \{T_{i,j}=0\}}
   &=&\prod _{j=1}^{\kappa'(S'')-1}\frac{(\sum_{i=1}^{\kappa(S'')-1}r''_i(\ell'')-\sum
   _{j^{\prime }=1}^{j-1}c''_{j^{\prime }}(\ell''))^{\underline {c''_j(\ell'')}}}{(N''(\ell'')-\sum
   _{j^{\prime }=1}^{j-1}c''_{j^{\prime }}(\ell''))^{\underline {c''_j(\ell'')}}}\nonumber\\
   &=&\prod _{j=1}^{\kappa'(S'')-1}\frac{(N''(\ell'')-\sum_{i=\kappa(S'')}^{\infty}r''_i(\ell'')-\sum
   _{j^{\prime }=1}^{j-1}c''_{j^{\prime }}(\ell''))^{\underline {c''_j(\ell'')}}}{(N''(\ell'')-\sum
   _{j^{\prime }=1}^{j-1}c''_{j^{\prime }}(\ell''))^{\underline {c''_j(\ell'')}}}\nonumber\\
   &=&1-o(1) \text{ as $\ell''\to\infty$},  \label{eq:step1}
   \end{eqnarray}
   since $c''_j(\ell'')=O(1)$ by~\eqref{eq:usefulfacts}, $\sum _{j^{\prime }=1}^{j-1}c''_{j^{\prime }}(\ell'')=O(1)$ for 
   $j\leq\kappa'(S'')-1$ by~\eqref{eq:sumc} and $\sum_{i=\kappa(S'')}^{\infty}r''_i(\ell'')=o(N''(\ell''))$ by property~(i) of $S''$. 
   Conditioned on the event studied in~\eqref{eq:step1}, the entry $T_{i,1}$ is not zero only if $%
   r''_i(\ell'')$ is not $o(N''(\ell''))$ (i.e., $i \leq \kappa(S'')-1$). 
   Thus, since there are $\kappa(S'')-1$ such rows and 
   $\lim_{\ell''\to\infty}c''_1(\ell)\geq \kappa(S'')$, with probability $1-o(1)$,
   there will be a non-binary entry $(i,1)$ with $i\leq \kappa(S'')-1$ for infinitely many values of $\ell''$.
   This gives that 
   $$
      \lim_{\ell''\to\infty} \pr{T \in \Omega_{\mathbf{r}'',\mathbf{c}''}}
      \leq 1-\lim_{\ell''\to\infty}\pr{\bigcap_{i=\kappa(S'')}^\infty \bigcap_{j=1}^{\kappa'(S'')-1}\{T_{i,j}=0\}}=0,
   $$
   which concludes the proof of the lemma.
\end{proof}

It remains to show that condition~\ref{it:n2} is sufficient.
We will first give three auxiliary lemmas and then will apply the lemmas to complete the proof of Proposition~\ref{prop:extra}. 
In the first lemma, 
we consider the case $\sum_{i=\kappa(S)}^\infty r_i(\ell)=\Omega(N(\ell))$ and show that 
this implies $\pr{T\in \Omega _{\mathbf{r},\mathbf{c}}}=\Omega(1)$. 
For this case, we have $\kappa(S)<\infty$.

\begin{lemma}
\label{lem:extrasuf} For any sequence $S=({\bf r}(\ell),{\bf c}(\ell),N(\ell))_\ell$, if the conditions in~\eqref{eq:condsec} are satisfied and 
$\sum_{i=\kappa(S)}^\infty r_i(\ell)=\Omega (N(\ell))$ as $\ell\to\infty$,  
then $\pr{T\in \Omega _{\mathbf{r},\mathbf{c}}}=\Omega(1)$.
\end{lemma}
\begin{proof}
   We first show that, with probability bounded away from $0$, $T_{i,j}=0$
   for each $i\in [1,\kappa(S)-1]$ and $j\in[1,\kappa'(S)-1]$. Let $X(\ell)=\sum_{i=1}^{\kappa(S)-1}r_{i}(\ell)$. Note
   that, by assumption, $N(\ell)-X(\ell)=\Omega (N(\ell))$. Therefore, there exists a constant 
   $\alpha \in (0,1)$ such that, for all $\ell$ sufficiently large, 
   \begin{equation}
   \pr{\bigcap_{i=1}^{\kappa(S)-1}\bigcap_{j=1}^{\kappa'(S)-1} \{T_{i,j}=0\}}
   =\prod_{j=1}^{\kappa^{\prime }(S)-1}\frac{(N(\ell)-X(\ell)-\sum_{j^{\prime
   }=1}^{j-1}c_{j^{\prime }}(\ell))^{\underline{c_{j}(\ell)}}}{(N(\ell)-\sum_{j^{\prime
   }=1}^{j-1}c_{j^{\prime }}(\ell))^{\underline{c_{j}(\ell)}}}\geq \alpha ,  \label{eq:ex1b}
   \end{equation}%
   since by~\eqref{eq:usefulfacts} we have $c_{j}=O(1)$ and $\sum_{j^{\prime }=1}^{j-1}c_{j^{\prime }}\leq
   c_{1}\kappa^{\prime }(S)=O(1)$ for $j\leq \kappa^{\prime }(S)-1$. Now, for all $%
   j\geq \kappa^{\prime }(S)$ and sufficiently large $\ell$, we have $c_{j}(\ell)\leq 1$,
   and all the entries $(i,j)$ for which $i \leq \kappa(S)-1$ and $j\geq
   \kappa^{\prime }(S)$ are binary with probability one. We can then conclude
   that the probability that all the entries for rows $i\leq\kappa(S)-1$ are
   binary is at least $\alpha$. Once we have sampled all the rows for which $%
   i\leq \kappa(S)-1$, we can then remove these rows and obtain new vectors 
   $\mathbf{r^{\prime }}(\ell)$ and $\mathbf{c^{\prime }}(\ell)$ for each $\ell$ such that $%
   \max_{i}r_{i}^{\prime }(\ell)=o(N(\ell))$ (by the definition of $\kappa(S)$) and $%
   \sum_{i}r_{i}^{\prime }(\ell)=\sum_{i=\kappa(S)}^\infty r_{i}(\ell)=\Omega (N(\ell))$. Note that the sequence 
   $(\mathbf{r^{\prime }}(\ell),\mathbf{c^{\prime }}(\ell),N(\ell))_\ell$ fall into the setting of
   Proposition~\ref{prop:suf}. Therefore, letting $T^{\prime }$ be a table generated
   from the row and column sums $\mathbf{r^{\prime }}(\ell)$ and $\mathbf{c^{\prime }}(\ell)$, we obtain 
   \begin{equation}
   \pr{T\in \Omega _{\mathbf{r},\mathbf{c}}}\geq \alpha \pr%
   {T^{\prime }\in \Omega _{\mathbf{r^{\prime }},\mathbf{c^{\prime }}}}=\Omega
   (1),  \label{eq:ex1a}
   \end{equation}%
   from (\ref{eq:ex1b}) and Proposition~\ref{prop:suf}.
\end{proof}

Now, in Lemma~\ref{lem:extrasuf2}, we consider the case where 
$\sum_{i=\kappa(S)}^\infty r_i(\ell)=o(N(\ell))$ for which $r_i(\ell)$ is either $\Omega(N(\ell))$ or $o(N(\ell))$ for all $i\geq 1$ as $\ell\to\infty$.
In this situation, we have that the condition $\limsup_{\ell\to\infty}c_1(\ell)< \kappa(S)$ is sufficient to obtain
$\pr{T\in \Omega _{\mathbf{r},\mathbf{c}}}=\Omega(1)$. The next lemma establishes this result when $\kappa(S)<\infty$.
\begin{lemma}
   \label{lem:extrasuf2} Consider a sequence $S=({\bf r}(\ell),{\bf c}(\ell),N(\ell))_\ell$ satisfying~\eqref{eq:condsec} and 
   for which $\kappa(S)<\infty$ and $r_i(\ell)$ is either $\Omega(N(\ell))$ or $o(N(\ell))$  as $\ell\to\infty$ for all $i\geq 1$.
   If $\sum_{i=\kappa(S)}^\infty r_i(\ell)=o(N(\ell))$
   and $\limsup_{\ell\to\infty}c_1(\ell)\leq \kappa(S)-1$, 
   then $\pr{T\in \Omega _{\mathbf{r},\mathbf{c}}}=\Omega(1)$.
\end{lemma}
\begin{proof}
   When $\sum_{i=\kappa(S)}^\infty r_i(\ell)=o(N(\ell))$, the same derivation that led to~\eqref{eq:step1} gives that
   $\pr{\bigcap_{i=\kappa(S)}^\infty\bigcap_{j=1}^{\kappa'(S)-1}\{T_{i,j}=0\}}=1-o(1)$. Therefore, when assigning a column
   $j\leq \kappa'(S)-1$,
   only the entries for rows
   $i\leq\kappa(S)-1$ will be non-zero with probability $1-o(1)$. 
   Since
   $\sum _{j=1}^{\kappa'(S)-1}c_j(\ell)\leq c_1(\ell)\kappa'(S)=O(1)$ by~\eqref{eq:usefulfacts} and $r_i(\ell)=\Omega(N(\ell))$ for all $i\leq \kappa(S)-1$,
   we have that 
   there is a constant $\alpha $ such that, for all large enough $\ell$, the probability that
   a given type-2 token for a column $j\leq \kappa'(S)-1$ is matched to a type-1 token for a given row $i\leq \kappa(S)-1$
   is at least $\alpha $ uniformly over all possible matchings for the other
   type-2 tokens from columns in $[1,\kappa'(S)-1]$. 
   Therefore,
   \begin{equation}
   \pr{\bigcap _{i=1}^{\kappa(S)-1}\bigcap _{j=1}^{\kappa'(S)-1}\{T_{i,j}\leq 1\}\;\middle \vert \;
   \bigcap _{i=\kappa(S)}^\infty\bigcap _{j=1}^{\kappa'(S)-1}\{T_{i,j}=0\}}\geq \prod _{j=1}^{\kappa'(S)-1}
   \binom{\kappa(S)-1}{c_j(\ell)}\alpha ^{c_j(\ell)}=\Omega (1).  \label{eq:step2}
   \end{equation}
   We then obtain that, with
   probability bounded away from zero, all the entries for columns in $[1,\kappa'(S)-1]$ are
   binary. For sufficiently large $\ell$, the remaining entries are all binary with probability one since, for
   $j\geq\kappa'(S)$, we have $\limsup_{\ell\to\infty} c_j(\ell)\leq1$. 
\end{proof}

Now we turn to the case where $\kappa(S)=\infty$.
\begin{lemma}\label{lem:kappainfty} 
   Consider a sequence $S=({\bf r}(\ell),{\bf c}(\ell),N(\ell))_\ell$ satisfying~\eqref{eq:condsec} and 
   for which $\kappa(S)=\infty$ and $r_i(\ell)=\Omega(N(\ell))$ for all $i\leq \limsup_{\ell \to\infty} c_1(\ell)$.
   Then $\pr{T\in \Omega _{\mathbf{r},\mathbf{c}}}=\Omega(1)$.
\end{lemma}
\begin{proof}
   We have $\sum_{j=1}^{\kappa'(S)-1}c_j(\ell)\leq c_1(\ell)\kappa'(S)=O(1)$ by~\eqref{eq:usefulfacts}. 
   Let $h=\limsup_{\ell\to\infty}c_1(\ell)$. Therefore, we have $c_j(\ell)\leq h$ for all $\ell$ sufficiently large and all $j$. 
   Also, since $r_1(\ell)=\Omega(N(\ell))$, there exists a constant $\epsilon>0$ so that 
   $\sum_{i=1}^h r_i(\ell) \geq r_1(\ell) \geq \epsilon N(\ell)$ for all sufficiently large $\ell$. 
   Then, we can use a similar derivation as in~\eqref{eq:ex1b} and show that 
   the probability that all type-2 tokens for columns $j<\kappa'(S)$ is matched to type-1 tokens for rows $i\leq h$ is at least 
   $\prod_{j=1}^{\kappa'(S)-1}\left(\tfrac{\epsilon N(\ell)-\sum_{j'=1}^{j-1}c_{j'}(\ell)}{N(\ell)}\right)^{c_{j}(\ell)}=\Theta(1)$. 
   Then, since $r_i(\ell)=\Omega(N(\ell))$ for all $i\leq h$ and $h\geq c_1(\ell)$ for all large enough $\ell$, a derivation 
   similar to~\eqref{eq:step2} shows that all the entries for columns $j<\kappa'(S)$ are binary with probability bounded away from zero. Since for the remaining columns
   $j\geq \kappa'(S)$ we have $c_j(\ell)\leq 1$ for all large enough $\ell$, all these entries are binary with probability one and we obtain $\pr{T\in \Omega _{\mathbf{r},\mathbf{c}}}=\Omega(1)$.
\end{proof}

Now we use the three lemmas above to show that condition~\ref{it:n2} is sufficient.
Let $S'=({\bf r}'(\ell'),{\bf c}'(\ell'),N'(\ell'))_{\ell'}$ be an arbitrary subsequence of $S=({\bf r}(\ell),{\bf c}(\ell),N(\ell))_\ell$. Then, by condition~\ref{it:n2}, it holds that 
either 
\begin{equation}
   \limsup_{\ell'\to\infty}\sum_{i=\kappa(S')}^\infty \tfrac{r_i'(\ell')}{N'(\ell')}>0 \quad\text{ or }\quad
   \limsup_{\ell'\to\infty} c_1'(\ell') < \kappa(S').
   \label{eq:cond2revisited}
\end{equation}
We want to show that
\begin{equation}
   \pr{T\in \Omega _{\mathbf{r}'',\mathbf{c}''}}=\Omega(1) \text{ for some subsequence $S''=({\bf r}''(\ell''),{\bf c}''(\ell''),N''(\ell''))_{\ell''}$ of $S'$}.
   \label{eq:extraclaim}
\end{equation}
Once this is established, since $S'$ is an arbitrary subsequence of $S$,
we have that~\eqref{eq:extraclaim} holds for every subsequence of $S$.
Then, using the subsubsequence principle of Lemma~\ref{lem:subsubsequence}, we obtain that 
$$
   \pr{T\in \Omega _{\mathbf{r},\mathbf{c}}}=\Omega(1),
$$
which concludes the proof of Proposition~\ref{prop:extra}.

Our goal now is to prove~\eqref{eq:extraclaim}. 
First, assume that $\kappa(S')<\infty$ 
and note that if 
$$
   \limsup_{\ell'\to\infty}\sum_{i=\kappa(S')}^\infty \frac{r'_i(\ell')}{N'(\ell')}
   > 0,
$$
then there exists a subsequence $S''=({\bf r}''(\ell''),{\bf c}''(\ell''),N''(\ell''))_{\ell''}$ of $S'$ such that 
$\sum_{i=\kappa(S')}^\infty r''_i(\ell'')=\Omega(N''(\ell''))$. Since $\kappa(S')\geq \kappa(S'')$ as observed in~\eqref{eq:kappamonotone},
we have that $\sum_{i=\kappa(S'')}^\infty r''_i(\ell'')=\Omega(N''(\ell''))$, 
and~\eqref{eq:extraclaim} follows from Lemma~\ref{lem:extrasuf}.
Therefore, we now assume that 
$$
   \sum_{i=\kappa(S')}^\infty r'_i(\ell') = o(N'(\ell')),
$$
which by the condition in~\eqref{eq:cond2revisited} gives that $\limsup_{\ell'\to\infty} c_1'(\ell') < \kappa(S')$.
By the definition of $\kappa(S')$, we have that $\limsup_{\ell'\to\infty} \tfrac{r_{i}'(\ell')}{N'(\ell')}>0$ for all $i< \kappa(S')$.
Since the $r_i'(\ell')$ are non-increasing with $i$ for any fixed $\ell$, there exists a subsequence $S''=({\bf r}''(\ell''),{\bf c}''(\ell''),N''(\ell''))_{\ell''}$ of $S'$ such that 
$$
   r''_i(\ell'') = \Omega(N''(\ell'')) \text{ for all $i< \kappa(S')$}.
$$
Also, for this subsequence, we have $\kappa(S'')=\kappa(S')$ and, consequently, 
$\sum_{i=\kappa(S'')}^\infty r''_i(\ell'') = o(N''(\ell''))$. Moreover, we have that $\limsup_{\ell''\to\infty}c_1''(\ell'') < \kappa(S'')$ by~\eqref{eq:kappamonotone} and the fact that 
$\limsup_{\ell'\to\infty} c_1'(\ell') < \kappa(S')$.
Then, from Lemma~\ref{lem:extrasuf2} we obtain~\eqref{eq:extraclaim}.

When $\kappa(S')=\infty$, we take $s=\limsup_{\ell'\to\infty}c_1'(\ell')$ and, since $c_1'(\ell')=O(1)$, we can find a subsequence $\hat S=({\bf\hat r}(\hat \ell),{\bf\hat c}(\hat \ell),\hat N(\hat \ell))_{\hat \ell}$ of 
$S'$ such that $\hat r_i(\hat \ell)=\Omega(\hat N(\hat \ell))$ as $\hat \ell\to\infty$ for all $i\leq s$. 
If $\kappa(\hat S)<\infty$, the argument above holds and we can find a subsequence $S''=({\bf r}''(\ell''),{\bf c}''(\ell''),N''(\ell''))_{\ell''}$ of 
$\hat S$ (and, consequently, also of~$S'$) for which~\eqref{eq:extraclaim} holds. Otherwise, we are in the setting of Lemma~\ref{lem:kappainfty}, which
establishes~\eqref{eq:extraclaim} for $S''\equiv \hat S$ and concludes the proof of Proposition~\ref{prop:extra}.

\section{Conclusions}

\label{sec:conc} We have characterized the input sequences for which the
configuration model is suitable for uniformly sampling and counting binary
contingency tables in optimal time (i.e., linear as a function of the number 
$N$ of unit entries in the table). Surprisingly, given known results for the
case of symmetric tables, having a bounded number of expected double edges
in the table is just a necessary condition for optimality but not
sufficient. It turns out that the full characterization for optimality in
the non-symmetric case relates to the behavior of very big rows (i.e. rows
of size $\Omega \left( N\right) $). Allowing such type of growth introduces
significant qualitative differences between symmetric and non-symmetric
tables. In turn, such differences give rise to technical challenges that are
not present in the symmetric case. Our results also have important practical
implications in applied settings that demand the development of
easy-to-implement and fast algorithms for uniform generation of binary
contingency tables.

We conclude by mentioning two open problems. Since, as we mentioned in
Section~\ref{sec:intro}, there is no need to employ a complicated sequential
importance sampling procedure when the conditions above hold, it is
interesting to know whether one can construct more specialized importance
sampling procedures to obtain low complexity polynomial algorithm when the
conditions in Theorem~\ref{thm:main} are not satisfied. Another open problem
consists in counting the number of (not necessarily binary) contingency
tables. In particular, it would be interesting to know whether a necessary
and sufficient condition like the one obtained in this paper can be derived
for that case. We remark that for general contingency tables, the
configuration model does not generate a table uniformly at random, which
makes the problem more challenging.

\section*{Acknowledgments}

We are thankful to Alistair Sinclair for bringing reference~\cite{janson2009}
to our attention and for helpful discussions. We are also thankful to the anonymous referees for their 
useful comments.
J.~Blanchet acknowledges support from the NSF foundation through the grants DMS-0806145, DMS-0846816 and DMS-1069064.
A.~Stauffer acknowledges support from a Fulbright/CAPES scholarship and NSF grants CCF-0635153 and DMS-0528488.

\bibliographystyle{plain}
\bibliography{binary}

\begin{thebibliography}{10}

\bibitem{barvinok2010}
A.~Barvinok.
\newblock On the number of matrices and a random matrix with prescribed row and
  column sums and 0-1 entries.
\newblock {\em Advances in Mathematics}, 224(1):316--339, 2010.

\bibitem{bayati2007}
M.~Bayati, J.H. Kim, and A.~Saberi.
\newblock A sequential algorithm for generating random graphs.
\newblock In {\em Proceedings of 11th Workshop on Randomization and Computation
  {(RANDOM)}}, pages 326--340, 2007.

\bibitem{besag1989}
J.~Besag and P.~Clifford.
\newblock Generalized {M}onte {C}arlo significance tests.
\newblock {\em Biometrika}, 76(4):633--642, 1989.

\bibitem{bezakova2007}
I.~Bez\'akov\'a, N.~Bhatnagar, and E.~Vigoda.
\newblock Sampling binary contingency tables with a greedy start.
\newblock {\em Random Structures and Algorithms}, 30(1-2):168--205, 2007.

\bibitem{bezakova2006}
I.~Bez\'akov\'a, A.~Sinclair, and D.~\v{S}tefankovi\v{c}.
\newblock Negative examples for sequential importance sampling of binary
  contingency tables.
\newblock In {\em Proceedings of the {E}uropean {S}ymposium on {A}lgorithms},
  pages 136--147, 2006.

\bibitem{bezakova2008}
I.~Bez\'akov\'a, D.~\v{S}tefankovi\v{c}, V.~Vazirani, and E.~Vigoda.
\newblock Accelerating simulated annealing for the permanent and combinatorial
  counting problems.
\newblock {\em {SIAM} Journal of Computing}, 37(5):1429--1454, 2008.

\bibitem{blanchet2009}
J.~Blanchet.
\newblock Efficient importance sampling for binary contingency tables.
\newblock {\em Annals of Applied Probability}, 19(3):949--982, 2009.

\bibitem{bollobas1980}
B.~Bollob\'as.
\newblock A probabilistic proof of an asymptotic formula for the number of
  labelled regular graphs.
\newblock {\em European Journal of Combinatorics}, 1:311--316, 1980.

\bibitem{canfield2008}
E.R. Canfield, C.~Greenhill, and B.D. {McKay}.
\newblock Asymptotic enumeration of dense 0-1 matrices with specified line
  sums.
\newblock {\em Journal of Combinatorial Theory, Series A}, 115:32--66, 2008.

\bibitem{chen2005}
Y.~Chen, P.~Diaconis, S.P. Holmes, and J.S. Liu.
\newblock Sequential {M}onte {C}arlo methods for statistical analysis of
  tables.
\newblock {\em Journal of the {A}merican {S}tatistical {A}ssociation},
  100(469):109--120, 2005.

\bibitem{dyer1997}
M.~Dyer, R.~Kannan, and J.~Mount.
\newblock Sampling contingency tables.
\newblock {\em Random Structures and Algorithms}, 10(4):487--506, 1997.

\bibitem{greenhill2008b}
C.~Greenhill, F.B. Holt, and N.~Wormald.
\newblock Expansion properties of a random regular graph after random vertex
  deletions.
\newblock {\em European Journal of Combinatorics}, 29:1139--1150, 2008.

\bibitem{greenhill2006}
C.~Greenhill, B.D. {McKay}, and X.~Wang.
\newblock Asymptotic enumeration of sparse 0-1 matrices with irregular row and
  column sums.
\newblock {\em Journal of Combinatorial Theory (Series A)}, 113:291--324, 2006.

\bibitem{janson2009}
S.~Janson.
\newblock The probability that a random multigraph is simple.
\newblock {\em Combinatorics, Probability and Computing}, 18(1-2):205--225,
  2009.

\bibitem{janson2000}
S.~Janson, T.~\L{}uczak, and A.~Ruci\'nski.
\newblock {\em Random Graphs}.
\newblock John Wiley \& Sons, 2000.

\bibitem{jerrum2004}
M.~Jerrum, A.~Sinclair, and E.~Vigoda.
\newblock A polynomial-time approximation algorithm for the permanent of a
  matrix with non-negative entries.
\newblock {\em Journal of the {A}ssociation for {C}omputing {M}achinery},
  51(4):671--697, 2004.

\bibitem{mckay1984}
B.D. {McKay}.
\newblock Asymptotics for 0-1 matrices with prescribed line sums.
\newblock In {\em Enumeration and Design}, pages 225--238. Academic Press,
  Canada, 1984.

\bibitem{mckay1985}
B.D. {McKay}.
\newblock Asymptotics for symmetric 0-1 matrices with prescribed row sums.
\newblock {\em Ars Combinatoria}, 19(A):15--25, 1985.

\bibitem{mitzenmacher2005}
M.~Mitzenmacher and E.~Upfal.
\newblock {\em Probability and Computing: randomized algorithms and
  probabilistic analysis}.
\newblock Cambridge University Press, Cambridge, UK, 2005.

\bibitem{oneil1969}
P.E. {O'Neil}.
\newblock Asymptotics and random matrices with row-sum and column
  sum-restrictions.
\newblock {\em Bulletin of the American Mathematical Society},
  75(6):1276--1282, 1969.

\bibitem{sinclair1989}
A.~Sinclair and M.~Jerrum.
\newblock Approximate counting, uniform generation and rapidly mixing markov
  chains.
\newblock {\em Information and Computation}, 82(1):93--133, 1989.

\bibitem{valiant1979}
L.~Valiant.
\newblock The complexity of computing the permanent.
\newblock {\em Theoretical Computer Science}, 8:189--201, 1979.

\end{thebibliography}
\end{document}